\documentclass[12pt,reqno]{amsart}

\usepackage{a4wide}
\usepackage{amsmath}
\usepackage{psfrag}
\usepackage{etex,hyperref}
\usepackage{amsthm}
\usepackage{amssymb,amsxtra}
\usepackage{appendix}
\usepackage[usenames]{color}
\usepackage{stmaryrd}
\usepackage{mathrsfs}
\usepackage{upgreek}
\usepackage{ytableau}
\usepackage{arydshln}
\usepackage{comment}
\usepackage{tikz-cd}
\usepackage{cleveref}
\usepackage{dynkin-diagrams}
\usepackage{float}
\usepackage{paralist}
\usepackage{graphicx}
\usepackage[all]{xy}
\usepackage{tikz}
\usetikzlibrary{matrix,fit}
\usepackage{pdflscape}
\usepackage{geometry}
 \allowdisplaybreaks

\newcommand{\NN}{\mathbb{N}}
\newcommand{\Z}{\mathbb{Z}}
\newcommand{\wt}{\widetilde}
\newcommand{\cS}{\mathcal{S}}

\newcommand{\sym}[1]{\mathfrak{S}_{#1}}

\newcommand{\Proj}{\mathsf{P}}

\renewcommand{\P}{\mathscr{P}}

\newcommand{\ccl}{\mathscr{C}}

\newcommand{\B}{\mathbb{B}}
\newcommand{\D}{\mathbb{D}}
\newcommand{\R}{\mathscr{R}}

\newcommand{\F}{{F}}

\newcommand{\FF}{\mathbb{F}}

\newcommand{\E}{\mathbb{E}}
\newcommand{\N}{\mathrm{N}}

\newcommand{\A}{\mathbb{A}}
\newcommand{\HH}{\mathbb{H}}
\newcommand{\I}{\mathbb{I}}

\newcommand{\Ext}{\mathsf{Ext}}

\newcommand{\modcat}{\textsf{mod} }
\newcommand{\Span}{\mathsf{Span}}

\newcommand{\cont}{\text{{\tiny $\#$}}}

\newcommand{\Q}{\mathbb{Q}}

\newcommand{\Fp}{\mathbb{F}_p}

\newcommand{\am}{n}

\newcommand{\Rad}{{\mathsf{Rad}}}
\newcommand{\Des}{\mathscr{D}}

\newcommand{\Hom}{{\mathsf{Hom}}}

\newcommand{\Ker}{\mathsf{ker}}
\newcommand{\soc}{\mathsf{Soc}}

\renewcommand{\c}{c}

\renewcommand{\sb}{\scalebox{.8}{$\bullet$}}

\renewcommand{\u}{t_0}
\renewcommand{\t}{t}

\newcommand{\red}{\color{red}}
\newcommand{\cO}{\mathcal{O}}

\theoremstyle{definition}

\newtheorem{theorem}{Theorem}[section]
\newtheorem{lemma}[theorem]{Lemma}

\newtheorem{proposition}[theorem]{Proposition}
\newtheorem{corollary}[theorem]{Corollary}
\newtheorem{remark}[theorem]{Remark}
\newtheorem{conjecture}[theorem]{Conjecture}

\newtheorem{example}[theorem]{Example}

\newtheorem*{main}{Theorem}

\numberwithin{equation}{section}

\begin{document}
\title{The Representation Type of the Descent Algebras}

\author{Karin Erdmann}
\address[K. Erdmann]{Mathematical Institute, University of Oxford, Radcliffe Observatory Quarter, Oxford OX2 6GG, United Kingdom.}
\email{erdmann@maths.ox.ac.uk}

\author{Kay Jin Lim}
\address[K. J. Lim]{Division of Mathematical Sciences, Nanyang Technological University, SPMS-04-01, 21 Nanyang Link, Singapore 637371.}
\email{limkj@ntu.edu.sg}

\begin{abstract} Schocker classified the representation type of the descent algebra of type $\A$ over any field of characteristic zero. In an earlier paper, the authors extended this classification for type $\A$ to fields of positive characteristic. In this paper, we complete the classification for all other types except for $\E_8$. The proof for type $\B$ is entirely theoretical, while some small cases in type $\D$ and the exceptional types require computer computation to determine their Ext-quivers.
\end{abstract}


\maketitle

\section{Introduction}

Every finite-dimensional algebra over an algebraically closed field can be classified by its representation type. It has finite representation type if it has only finitely many indecomposable modules up to isomorphism; otherwise, it has infinite representation type. In the case of infinite representation type, a result of Drozd \cite{Drozd:1980} shows that it can be further classified as tame or wild. This classification problem has been studied extensively in the literature \cite{Auslander/Reiten/Smalo:1995a, Dlab/Ringel:1976a, Drozd:1980, Erdmann:1990a, Gabriel:1972a, Ringel:1975a}. Much attention has been given to basic algebras, since every finite-dimensional algebra is Morita equivalent to such an algebra. Furthermore, by a result of Gabriel \cite{Gabriel:1979a}, every finite-dimensional basic algebra is isomorphic to a quotient of the path algebra defined by its Ext-quiver by some admissible ideal. Therefore, the Ext-quiver plays a crucial role in studying the representation type of an algebra.

In this paper, we focus on the Solomon descent algebra of a (finite) Coxeter group. The descent algebra was first defined by Solomon \cite{Solomon:1976a} as a subalgebra of the group algebra of the Coxeter group. Together with the work of Atkinson--van Willigenburg \cite{Atkinson/vanWilligenburg:1997a} and Atkinson--Pfeiffer--van Willigenburg \cite{Atkinson/Pfeiffer/vanWilligenburg:2002a}, it is now known that every descent algebra (over an arbitrary field) is basic, as well as the understanding of its Jacobson radical, the classification of its simple modules, and decomposition matrix. Moreover, the descent algebra has a distinguished  basis arising from the length function of the Coxeter system. The structure constants for the bases in types $\A$, $\B$, and $\D$ have combinatorial descriptions \cite{Bergeron/Bergeron:1992, Bergeron/vanWilligenburg:1998, Garsia/Reutenauer:1989a}.

Let $p$ be a prime, or $p=\infty$, the  characteristic of the field over which the descent algebra is defined.  In type $\A$, the algebra is related to the theory of free Lie algebras. In this connection, when $p=\infty$, the decomposition problem for the descent algebra was studied in \cite{Garsia/Reutenauer:1989a} and later extended to the other types in \cite{Bergeron/Bergeron:1992, Bergeron/Bergeron/Howlett/Taylor:1992}. This leads to another basis for the descent algebra when $p=\infty$, which has been used to obtain a complete set of primitive orthogonal idempotents. When $p<\infty$, no closed form is available, but primitive orthogonal idempotents can still be constructed \cite{Benson/Lim:2025a, Lim:2023a}. When $p=\infty$, another important basis for the type $\A$ descent algebra is the set of higher Lie idempotents introduced by Blessenohl--Laue \cite{Blessenohl/Laue:1996a}. Using these elements, they were able to describe the structure of the projective indecomposable modules and determine the Cartan numbers \cite{Blessenohl/Laue:2002a}. Together with a result of Atkinson--Pfeiffer--van Willigenburg \cite{Atkinson/Pfeiffer/vanWilligenburg:2002a}, the Cartan numbers for the type $\A$ descent algebras over arbitrary fields are known. Furthermore, the result of Blessenohl--Laue determines the Ext-quivers in that case, and in \cite{Schocker:2004a}, Schocker used this to determine their representation type. 

The Ext-quiver and the Cartan numbers play a crucial role in determining the representation type of an algebra. Extending Schocker’s work, Saliola \cite{Saliola:2008a} computed the Ext-quivers for the descent algebras of the Coxeter groups $W=\A_n$ and $W=\B_n$ over fields of characteristic $p$ such that $p\nmid |W|$. In general, especially when $p\mid |W|$, the Ext-quiver remains unknown. However, the construction of primitive orthogonal idempotents in \cite{Benson/Lim:2025a, Lim:2023a} is sufficient to compute the Ext-quivers for small cases. This method was used  in \cite{Erdmann/Lim:2025} by the authors to identify the representation type of the type $\A$ descent algebras by reducing the classification problem to algebras of small degree.

Continuing   the  classification of the representation type for type $\A$, we complete in this paper the classification for all other types (except for $\E_8$) over arbitrary fields. We summarise the results below, including the type $\A$ case obtained by Schocker \cite{Schocker:2004a} and by the authors \cite{Erdmann/Lim:2025}, as well as Theorems \ref{T:RepTypeB}, \ref{T:RepTypeD}, \ref{T:RepTypeI}, and \ref{T:RepTypeEFH} in this paper.

\begin{main}[Classification of the Representation Type of the Descent Algebra] Let $F$ be a field of characteristic $p$ where $p$ is either a prime or $p=\infty$ and $\Des_F(W)$ be the descent algebra of the Coxeter group $W$ over $F$. 
\begin{enumerate}[(1)]
\item Let $n\geq 2$. The algebra $\Des_F(\A_{n-1})$ has finite representation type if and only if either
\begin{enumerate}[(i)]
  \item $p=2$ and $n\leq 3$,
  \item $p=3$ and $n\leq 4$, or
  \item $p\geq 5$ and $n\leq 5$.
\end{enumerate} Otherwise, it has wild representation type.
\item Let $n\geq 2$. The algebra $\Des_F(\B_n)$ is
\begin{enumerate}[(i)]
\item finite type if $p\geq 3$ and $n\leq 4$,
\item tame if $p=2=n$,
\item wild otherwise.
\end{enumerate}
\item Let $n\geq 4$. The descent algebra $\Des_F(\D_n)$ is
\begin{enumerate}[(i)]
  \item tame if $p\geq 5$ and $n=4$,
  \item wild otherwise.
\end{enumerate}
\item The algebra $\Des_F(\I_n)$ is tame if $p=2\mid n$ and has finite representation type otherwise.
\item Let $X$ be either $\FF_4$, $\HH_3$, $\HH_4$, $\E_6$ or $\E_7$, and let $A=\Des_F(X)$.
\begin{enumerate}[(i)]
\item If $X=\FF_4$, then $A$ is tame if $p\geq 5$, and wild otherwise.
\item If $X=\HH_3$, then $A$ is wild if $p=2$, and tame otherwise. 
\item If $X=\HH_4$, then $A$ is wild.
\item If $X=\E_6$, then $A$ is wild.
\item If $X=\E_7$, then $A$ is wild.
\end{enumerate}  
\end{enumerate}
\end{main}

To prove our results for the infinite families of Coxeter groups of types $\B$ and $\D$, we make use of the surjective algebra homomorphisms defined by Aguiar--Bergeron--Nyman \cite{Aguiar/Bergeron/Nyman:2004}. More precisely, in type $\B$, using the homomorphisms relating descent algebras of type $\B$ and the known Ext-quivers, we reduce the problem to  the cases $\B_n$ with $n\leq 5$. In type $\D$, the Ext-quiver is less well understood. However, using homomorphisms relating types $\B$ and $\D$, together with our classification for type $\B$, we again reduce to the cases $\D_n$ with $n\leq 7$. To handle these small cases, we use Magma to compute their Ext-quivers, based on the construction of idempotents from \cite{Benson/Lim:2025a}. For type $\I$, the descent algebra is only 4-dimensional, making direct computation feasible. For the finite types $\FF_4$, $\HH_3$, $\HH_4$, $\E_6$, and $\E_7$ (except $\E_8$), we again compute their Ext-quivers using Magma.

We organise the paper as follows. In the next section, we collate the necessary background on finite-dimensional algebras, especially the theory of basic algebras. In Section \ref{S:Descent}, we specialise to the descent algebras of Coxeter groups, with a particular focus on types $\B$ and $\D$. We introduce the relevant combinatorial descriptions of their simple modules and discuss the surjective algebra homomorphisms that will be used in subsequent sections. Their Ext-quivers are examined further in Section \ref{S:ExtBandD}. The classification of the representation type for types $\B$, $\D$, $\I$, and the finite types $\FF_4$, $\HH_3$, $\HH_4$, $\E_6$, and $\E_7$ is presented in Sections \ref{S:RepTypeB}, \ref{S:RepTypeD}, \ref{S:RepTypeI}, and \ref{S:RepTypeEFH}, respectively.

\section{Preliminaries}

We introduce some basic notation we use throughout the paper. Let $\NN,\NN_0$ denote the set of positive and non-negative integers respectively. For any integers $a\leq b$, we write $[a,b]$ for the set of integers $i$ satisfying $a\leq i\leq b$. We fix an algebraically closed field $F$  of characteristic $p$ where $p$ is either a prime or $p=\infty$. We remark that $\mathbb{F}_p$ is a splitting field for the descent algebras but we require the algebraically closed assumption for general theory of finite-dimensional algebra. For background 
on representation theory of algebras,  we refer to \cite{Auslander/Reiten/Smalo:1995a, Benson:1991}. 


\subsection{Quiver and Path algebra} A (finite) quiver $Q$ is a directed graph $Q=(Q_0,Q_1)$ where $Q_0$ and $Q_1$ are the finite vertex and finite arrow sets respectively.  For each arrow $\gamma\in Q_1$, we write $s(\gamma)$ and $t(\gamma)$ for the starting and terminating vertices of $\gamma$ respectively. The underlying graph of $Q$ is the undirected graph of $Q$ by forgetting the orientations of the arrows in $Q$. 

Given a quiver $Q$, a path of length $\ell\in\NN_0$ in $Q$ is a finite sequence $\alpha=\alpha_\ell\cdots\alpha_1$ where $\alpha_1,\ldots,\alpha_\ell$ are arrows and, for each $i\in [1,\ell-1]$, we have $t(\alpha_i)=s(\alpha_{i+1})$. In this case, we write $s(\alpha)=s(\alpha_1)$ and $t(\alpha)=\alpha_\ell$. For each $v\in Q_0$, we write $e_v$ for the path of length 0 where $s(e_v)=t(e_v)=v$. Let $\beta$ be another path in $Q$. If $t(\alpha)=s(\beta)$, then we write $\beta\alpha$ for the composition of these two paths. Notice that our convention of path composition is read from right to left. 

Given a  quiver $Q$ and a field $F$, the path algebra $FQ$ is defined as follows. As an $F$-vector space, $FQ$ has basis consisting of all paths in $Q$. The multiplication $\beta\cdot\alpha$ of two paths $\alpha$ and $\beta$ is defined as the path $\beta\alpha$ if $s(\beta)=t(\alpha)$ and 0 otherwise. As such, it is an  associative $F$-algebra with the identity $\sum_{v\in Q_0}e_v$. Furthermore, the set $\{e_v:v\in Q_0\}$ is a complete set of primitive orthogonal idempotents of $FQ$ and $FQ$ is finite-dimensional if and only if $Q$ does not contain an oriented cycle. 

\subsection{Ext-quiver}

Let $A$ be a finite-dimensional $F$-algebra. For each $A$-module $V$ (finite-dimensional over $F$) and $j\in\NN_0$, we write $\Rad^j(A)$ and $\Rad^j(V)$ for the $j$th Jacobson radical of $A$ and $j$th radical power of $V$ respectively. Moreover, we write $\Proj(V)$ for the projective cover of $V$ and $\soc(V)$ for the socle of $V$. Suppose that $S_1,\ldots,S_m$ are all the simple $A$-modules up to isomorphism and let $P_1,\ldots,P_m$ be their respective projective covers.  Furthermore, we let $e_1,\ldots,e_m$ be a complete set of primitive orthogonal idempotents of $A$ such that $P_i\cong Ae_i$ for each $i\in [1,m]$. The Ext-quiver $Q$ of $A$ is the quiver with vertices labelled by the simple $A$-modules, in this case, the numbers $1,\ldots,m$, and, for each $i,j\in [1,m]$, the number of arrows from $i$ to $j$ is given by the dimension of $\Ext^1_A(S_i,S_j)$. Recall that the space $\Ext^1_A(S_i,S_j)$ is isomorphic with \[\Hom_A(P_j,\Rad(P_i))/ \Hom_A(P_j,\Rad^2(P_i))\cong e_j\Rad(A)e_i/e_j\Rad^2(A)e_i.\] In other words, the number of arrows from $i$ to $j$ is the number of $S_j$ appearing as a direct summand in $\Rad(P_i)/\Rad^2(P_i)$.

For our purpose, we need to  change of base ring and understand its effect on the Ext-quiver. We record a lemma which shall be used later. The proof follows from \cite[Lemma 1.5.2]{Cline/Parshall/Scott:1996a} and we refer the reader to \cite[\S2.8]{Erdmann/Lim:2025} for further discussion. 

\begin{lemma}\label{L:Extmodp} Let $(K,\cO,k)$ be a $p$-modular system and $B$ be an $\cO$-algebra which is finitely generated and $\cO$-free. Suppose that $V,T$ are finitely generated $B$-modules that are also $\cO$-free. If $\Ext^1_{B\otimes_\cO K}(V\otimes_\cO K, T\otimes_\cO K)\neq 0$, then $\Ext^1_{B\otimes_\cO k}(V\otimes_\cO k, T\otimes_\cO k)\neq 0$.
\end{lemma}

Furthermore, we want to exploit surjective algebra homomorphisms between different basic algebras, and we will apply the following general observation.

\begin{lemma}\label{L:surjExt}
Assume $\Lambda$ and $\Gamma$ are two $F$-algebras, with Ext-quivers $Q_{\Lambda}$ and $Q_{\Gamma}$. 
Suppose $\beta: \Lambda\to \Gamma$ is a surjective algebra homomorphism. Then $Q_{\Gamma}$ can be identified as a subquiver of 
$Q_{\Lambda}$.
\end{lemma}
\begin{proof}  The algebra map $\beta$ induces the inflation map $\beta^*$ from $\Gamma$-modules to $\Lambda$-modules, that is a module for $\Gamma$ is viewed via $\beta$ as a module for $\Lambda$. So $\beta^*(S)$ is simple for any simple $\Gamma$-module $S$. Furthermore, $\beta^*(S_1)$ and $\beta^*(S_2)$ are not isomorphic if $S_1$ and $S_2$ are not isomorphic. So we can identify the vertices of $Q_{\Gamma}$ with a subset of the vertices of $Q_{\Lambda}$.  


Recall that arrows of an Ext-quiver are in bijection with a basis of the quotient modulo the second power of the radical. We want to show that, as $\Lambda$-modules, $\Gamma/\Rad^2(\Gamma)$ is isomorphic to a factor of $\Lambda/\Rad^2(\Lambda)$.
Consider $U=\Gamma/\Rad^2(\Gamma)$ and view this as a module for $\Gamma$. 
Let $\eta: \Lambda \to U$ be the composition of $\beta$ followed by the natural map $\Gamma \to U$. Let $\Lambda'=\Ker(\eta)$.
The $\Gamma$-module $U$ has radical length two by definition. Then also the $\Lambda$-module $\beta^*(U)$ has radical length two (we use the first part of the proof).  Therefore $\Rad^2(\Lambda)$  is contained in the kernel of $\eta$. 
Now, $\beta^*(U)$ is isomorphic to $\Lambda/\Lambda'$, and this is then a factor module of $\Lambda/\Rad^2(\Lambda)$.
\end{proof}

\subsection{Representation type}

Let $A$ be a finite-dimensional algebra over the field $F$. The algebra $A$  has finite representation type if there are only finite number of  indecomposable $A$-modules, up to isomorphism. For example, semisimple algebras have finite representation type.  The algebra  has tame representation type if it is not of finite representation type but, up to isomorphism, for each $d\in\NN$, almost all (except finitely many) indecomposable $A$-modules of dimension $d$ belong to 
finitely many  one-parameter families. Let 
$F\langle X,Y\rangle$ be the free $F$-algebra on two variables $X, Y$. The algebra $A$ has wild representation type if there exists a finitely generated ($A$-$F\langle X,Y\rangle$)-bimodule $M$ such that $M$ is free as right $F\langle X,Y\rangle$-module and the functor $M\otimes_{F\langle X,Y\rangle}-$ from  $F\langle X,Y\rangle$-\modcat  to $A$-\modcat preserves indecomposability and isomorphism classes. We have the following theorem of Drozd (see also \cite{CB}).

\begin{theorem}[\cite{Drozd:1980}]\label{T:trichotomy} An algebra which is not of  finite type is either tame or wild but not both.
\end{theorem}

For the purpose of identifying the  representation type, the following lemmas will be useful. 

\begin{lemma}\label{L: surj alg} Let $A,B$ be $F$-algebras. Suppose that there is a surjective algebra homomorphism $\phi:A\to B$. If $B$ has infinite (respectively, wild) representation type then $A$ has infinite (respectively, wild) representation type.
\end{lemma}

\begin{lemma}[{\cite[Proposition 1.2]{Erdmann/Holm/Iyama/Schroer:2004a}}]\label{L:projinj} If $B=A/\soc(P)$ where $P$ is both a projective and injective $A$-module, then $A$ and $B$ have the same representation type.
\end{lemma}

In general, it is difficult to determine the representation type of an algebra. However, there are cases in which it is sufficient to exploit  the separated quivers of the Ext-quiver. Given a quiver $Q=(Q_0,Q_1)$. The separated quiver of $Q$ is the quiver with the vertex set $\{i:i\in Q_0\}\cup\{i':i\in Q_0\}$ and arrow set consisting of $i\to j'$ for each arrow $i\to j$ in $Q_1$. In the case when the algebra $A$ is 2-nilpotent, that is,  $\Rad^2(A)=0$, $A$ is stably equivalent to an algebra with Ext-quiver the separated quiver of the Ext-quiver of $A$. A general result of Krause \cite{Krause:1997} implies the following statement.  

\begin{theorem}\label{T:separatedquiver} Let $A$ be an  $\F$-algebra, $Q$ be the Ext-quiver of $A$ and suppose that $\Rad^2(A)=0$.
\begin{enumerate}[(i)]
\item {{\cite[X, Theorem 2.6]{Auslander/Reiten/Smalo:1995a}}} Then $A$ has finite type if and only if the separated quiver of $Q$ is a finite union of Dynkin diagrams.
\item {\cite{Dlab/Ringel:1976a,Donovan/Freislich:1973}} Then $A$ is tame if and only if the separated quiver of $Q$ is a finite union of Dynkin diagrams and (at least one) extended Dynkin diagrams.
\end{enumerate}
\end{theorem}

In this paper, a particular application of Lemma \ref{L: surj alg} and Theorem \ref{T:separatedquiver} is given as follows.

\begin{corollary}\label{C:wildquotient} If the separated quiver of $A$ (and hence of $A/\Rad^2(A)$) contains a component other than a Dynkin or extended Dynkin diagram, then $A$ is wild.
\end{corollary}
\begin{proof} Let $B=\Rad^2(A)$. We consider the canonical surjection $\pi:A\to B$. The Ext-quivers of $A$ and $B$ are identical. Since $B$ is 2-nilpotent, $B$ is not finite nor tame by our assumption using Theorem \ref{T:separatedquiver}. So it is wild by Theorem \ref{T:trichotomy}. Hence $A$ is wild by Lemma \ref{L: surj alg}. 
\end{proof}

Furthermore, if the algebra is the path algebra $FQ$, the representation type depends on the quiver $Q$.  This is  described in the following two theorems.

\begin{theorem}[\cite{Gabriel:1972a}]\label{T: Gab} The path algebra $\F Q$ has finite type if and only if the underlying graph of $Q$ is a disjoint union of Dynkin diagrams of types $\mathbb{A}$, $\mathbb{D}$ or $\mathbb{E}$. In this case, up to isomorphism, the indecomposable $\F Q$-modules are in one-to-one correspondence with the positive roots of the associated root system.
\end{theorem}

\begin{theorem}[\cite{Dlab/Ringel:1976a,Donovan/Freislich:1973}]\label{T:tame} Suppose that $Q$ is connected without oriented cycles. Then $\F Q$ is tame if and only if the underlying graph of $Q$ is a (simply laced) extended Dynkin diagram.
\end{theorem}

\subsection{Basic algebra}

An $F$-algebra is basic if all simple $A$-modules are one-dimensional. For example, every path algebra is basic. The notion of basic algebra is crucial because  every finite-dimensional $F$-algebra is Morita equivalent to a basic $F$-algebra. Furthermore, the structure of such a basic algebra can be described using the well-known Gabriel's theorem. More precisely, it is the quotient of the path algebra of its Ext-quiver by some admissible ideal. 


To describe the theorem of Gabriel, we need the following notation. The ideal of a path algebra $FQ$ generated by all arrows is denoted by $FQ^+$. In other words, $FQ^+$ has a basis consisting of all paths of positive lengths. Therefore, for each $j\in \NN_0$, $(FQ^+)^j$ is the ideal consisting of all paths of lengths at least $j$. 

\begin{theorem}[\cite{Gabriel:1979a}]\label{T:Gabriel} Let $A$ be a basic $F$-algebra and $Q$ be the Ext-quiver of $A$. Then $A\cong F Q/I$ where $I$ is an ideal of $FQ$ such that $(F Q^+)^n\subseteq I\subseteq (F Q^+)^2$ for some integer $n\geq 2$.
\end{theorem}


\section{Descent Algebra}\label{S:Descent}

In this section, we review the Solomon descent algebra for a Coxeter group. We refer the reader to the references \cite{Atkinson/Pfeiffer/vanWilligenburg:2002a, Atkinson/vanWilligenburg:1997a,Geck/Pfeiffer:2000, Humphreys:1990, Solomon:1976a} for further details. After introducing the background knowledge for general finite Coxeter groups, we focus particularly on types $\B$ and $\D$. 

\subsection{General theory} Let $(W,S)$ be a finite Coxeter system and $\ell:W\to\NN$ be the length function with respect to $S$. For a subset $J\subseteq S$, let $W_J$ be the subgroup of $W$ generated by $J$, 
a  (standard) parabolic subgroup.  Let  $X_J$ be the distinguished set of left coset representatives of $W_J$ in $W$ consisting of minimal length representatives, that is, \[X_J=\{w\in W: \text{$\ell(ws)>\ell(w)$ for all $s\in J$}\}.\] The set $X_J^{-1}=\{w^{-1}:w\in X_J\}$ is thus the distinguished set of minimal length right coset representatives of the parabolic subgroup $W_J$ in $W$. For another subset $K\subseteq S$, $X_{JK}:=X_J^{-1}\cap X_K$ is therefore a set of double coset representatives of $(W_J,W_K)$ in $W$.

Let $J\subseteq S$ and we define the element $x_J=\sum_{w\in X_J} w \in\Z W$. Solomon \cite{Solomon:1976a} showed that, over $\Z$, for any two subsets $J,K$ of $S$,
\begin{equation}\label{Eq:DesM}
x_Jx_K=\sum_{L\subseteq S}a_{JKL}x_L
\end{equation} where $a_{JKL}$ is the number of elements $w\in X_{JK}$ such that $L=w^{-1}Jw\cap K$. In other words, the $\Z$-span of the set $\Omega:=\{x_J:J\subseteq S\}$ is a $\Z$-subalgebra of the group algebra $\Z W$ and it is called the descent algebra of $W$ (over $\Z$). We denote it by $\Des_\Z(W)$. Furthermore, $\Omega$ is a free $\Z$-basis for $\Des_\Z(W)$. Since the structure constants $a_{JKL}$ are integers, for any integral domain $R$, considering $a_{JKL}\cdot 1_R$, we obtain the descent algebra \[\Des_R(W):=R\otimes_\Z\Des_\Z(W)\] which is considered as a subalgebra of the group algebra $RW$.



For each $J\subseteq S$, let $\varphi_J$ be the permutation character of the induced module $\Z_{W_J}{\uparrow^W}$ from the trivial module for $W_J$ to $W$, which the value $\varphi_J(x)$ counts the fixed points of $x\in W$ in the action on the cosets  $W/W_J$. So $\varphi_J$ takes integer values on the elements of $W$. Let $\ccl(W)$ be the $\Z$-span of the set $\{\varphi_J:J\subseteq S\}$. We have the well-known Mackey formula
\begin{equation}\label{Eq:Mackey}
\varphi_J\varphi_K=\sum_{L\subseteq S}a_{JKL}\varphi_L
\end{equation} where $a_{JKL}$'s are precisely the integers we have obtained earlier in Equation \ref{Eq:DesM}. 
Equations \ref{Eq:DesM} and \ref{Eq:Mackey} show that there is  a homomorphism of $\Z$-algebras \[\theta:\Des_\Z(W)\to  \ccl(W)\] given by $\theta(x_J)=\varphi_J$ (see 
\cite{Solomon:1976a}). Solomon also showed that the kernel of $\theta$ is the radical of $\Des_{\Z}(W)$, and that it is spanned by the set of all $x_K-x_J$ where $J$ and $K$ are conjugate in $W$.

\medskip

Let $\overline{\ccl(W)}$ be the $F$-span of the set $\{\varphi^{F}_J:J\subseteq S\}$ where \[\varphi^{F}_J(x)=\varphi_J(x)\cdot 1_F\] for each $x\in W$.
 Tensoring with $F$ and reducing modulo $p$, the map $\theta$ gives rise to an $F$-algebra homomorphism $\theta_F:\Des_F(W)\to \overline{\ccl(W)}$ where $\theta_F(x_J)=\varphi^{F}_J$. The algebra $\overline{\ccl(W)}$ is
 always commutative semisimple.

 The following generalisation of Solomon's description of the radical $\Rad(\Des_F(W))$  was proved by  Atkinson--Pfeiffer--van Willigenburg.

\begin{theorem}[{\cite[Theorem 3]{Atkinson/Pfeiffer/vanWilligenburg:2002a} and \cite[Theorem 3]{Solomon:1976a}}]\label{T: Sol Radical} Let $(W,S)$ be a Coxeter system. Then $\Rad(\Des_F(W))$ is spanned by elements $x_J - x_K$ such that $J $
and $K$ are conjugate in $W$, and together with elements $x_L$ such that $p\mid [\N_W(W_L):W_L]$. Furthermore, $\Rad(\Des_F(W))=\Ker(\theta_F)$. In particular, the simple modules of $\Des_F(W)$  are 1-dimensional.
\end{theorem} 
 
\bigskip

The task is now to understand the simple modules, and the essential input are the permutation characters $\varphi_J$ for $J\subseteq S$, which counts the number of fixed points on the coset spaces $W/W_J$. 
For subsets $J$ and $K$ of $S$, let 
$\beta_{JK} = |{\rm Fix}_{W/W_J}(W_K)|$.  The following shows that these can be computed using a Coxeter element $c_K$ of $W_K$,  which is the product of all elements of $K$ in some order.

\begin{lemma}
[{\cite[Lemma 1]{Atkinson/Pfeiffer/vanWilligenburg:2002a}}]\label{L:beta}
Let $(W,S)$ be a Coxeter system and $J,K$ be subsets of $S$. We have \ 
\begin{align*}
\beta_{JK}& = [\N_W(W_J):W_J]\cdot |\{W^\omega_J:\omega\in W,\ W_K\subseteq W_J^\omega\}| \cr
&= |\{ w\in X_J^{-1}\cap X_K: J^w\cap K=K\}|\cr
&= a_{JKK}\cr
&= \varphi_J(c_K).
\end{align*}
Furthermore, $\beta_{JJ} =  [\N_W(W_J):W_J]\neq 0$ and $\beta_{JJ}$ divides $\beta_{JK}$ for each $K\subseteq S$. Also, if $\beta_{JK}\neq 0$, then $K\subseteq_W J$. 
\end{lemma}


Let $\R$ be a set of representatives for the conjugacy classes of subsets of $S$. Then the parabolic table of marks is defined as the matrix $M^c(W) = (\beta_{JK})_{J, K\in \R}$. By \cite{Atkinson/Pfeiffer/vanWilligenburg:2002a}, with suitable ordering on $\R$, $M^c(W)$ is lower triangular of rank $|\R|$. We fix such a total order for $\R$ throughout.  For the number $p$ (where $p$ is either a prime or $\infty$), let $\R_p$ be the subset of $\R$ consisting of $J$ such that $p\nmid [N_W(W_J):W_J]$.  Then the size of $\R_p$ is the rank of $M^c(W)$ modulo $p$. Notice that $\R_\infty=\R$.

Recall the map $\theta_F$ defined earlier, and  we set $\theta= \theta_{\Q}$. For each $K\in \R_p$, define the map 
$\lambda_K: \Des_{F}(W) \to  F$ by, for any $x\in \Des_F(W)$,  
\[\lambda_K(x) = \theta_F(x)(c_K).\] Since $\theta_F$ is a homomorphism, it follows that $\lambda_K$ is a homomorphism, and hence gives rise to a 1-dimensional representation of $\Des_F(W)$. 

The map $\lambda_K$ is completely determined by its values on the basis elements $x_J$, since
$\lambda_K(x_J) = \theta_F(x_J)(c_K) = \varphi_J^F(c_K) = \beta_{JK}\cdot 1_F$, and these values of $\lambda_K$ form the column of the matrix $M^c(W)$ indexed by $K$. We denote the one-dimensional simple module over $F$ induced by $\lambda_K$ by $M_{K, F}$. We obtain the following:

\begin{lemma}[{\cite[Lemmas 4]{Atkinson/Pfeiffer/vanWilligenburg:2002a}}]\label{L: matrix M_F} For each $K\in \R_p$, the simple $\Des_F(W)$-module $M_{K,F}$ is one-dimensional with $x\in \Des_F(W)$ acts via the multiplication by $\theta_F(x)(c_K)\in F$ where $c_K$ is a Coxeter element of $W_K$. 
\end{lemma}

For any $K\in\R$, let $M_{K,\Z}$ be the $\Z$-free of rank one $\Des_\Z(W)$-module where $x\in \Des_\Z(W)$ acts via the multiplication by $\theta(x)(c_K)\in\Z$. So $M_{K,\Z}\otimes_\Z F\cong M_{K,F}$ if $K\in \R_p$. Also, by Lemma \ref{L: matrix M_F}, $M_{K,\Z}\otimes_\Z F\cong M_{K',\Z}\otimes_\Z F$ if and only if $\varphi^F_J(c_K)=\varphi^F_J(c_{K'})$ for all $J\subseteq S$.

The total order for $\R_\infty=\R$ gives rise to a total order for the subset $\R_p$. The decomposition matrix $D$ of the descent algebra of $W$ is the matrix, row-wise (respectively, column-wise) labelled by $\R_\infty$ (respectively, $\R_p)$, with the entries \[d_{J,K}=\left \{\begin{array}{ll}1&\text{if $\varphi^F_L(c_J)=\varphi^F_L(c_K)$ for all $L\subseteq S$,}\\ 0&\text{otherwise,}
\end{array}\right .\] where $J\in \R_\infty$ and $K\in\R_p$. The Cartan matrix for the descent algebra is related to the decomposition matrix given by the following result of Atkinson--Pfeiffer--van Willigenburg. 

\begin{theorem}[{\cite[Theorem 8]{Atkinson/Pfeiffer/vanWilligenburg:2002a}}]\label{T: APW} Fix a total order for the set $\R$ (and hence for the subset $\R_p$). Let $C$ be the Cartan matrix of $\Des_\Q(W)$ and $D$ be the decomposition matrix. Then the Cartan matrix of $\Des_{\mathbb{F}_p}(W)$ is $\wt{C}=D^\top CD$.
\end{theorem}

Since the simple modules for the descent algebra have $\Z$-forms of rank one and they are defined over arbitrary field, we remark that the equation $\wt{C}=D^\top CD$ in Theorem \ref{T: APW} applies when $C,\wt{C}$ are the Cartan matrices for the descent algebras of $W$ defined over arbitrary fields of characteristic $\infty$ and $p$ respectively. 

\medskip

In this paper, we use the Ext-quivers of the descent algebras to determine their representation type. For type $\A$ and $p=\infty$ case, the Ext-quiver has been obtained by Schocker \cite{Schocker:2004a}. Subsequently, Saliola determined the Ext-quivers for type $\A$ and $\B$ cases provided that $p\nmid |W|$. Otherwise, in general, determination of the Ext-quiver remains open. However, a recent work of Benson and the second author \cite{Benson/Lim:2025a} shows the following. In the corollary, $n_W$ is certain number defined in \cite[Definition 9.9]{Benson/Lim:2025a} which we shall not repeat here.

\begin{corollary}\label{C:BensonLimExtQuiver} Let $(W,S)$ be a Coxeter system and let $Q_\infty$ be the Ext-quiver of the descent algebra $\Des_{\mathbb{C}}(W)$. If $p\nmid n_W$, then the Ext-quiver of $\Des_F(W)$ is $Q_\infty$. 
\end{corollary}
\begin{proof} By \cite[Theorem 9.15]{Benson/Lim:2025a}, when $p\nmid n_W$, \cite[Theorem 2.3(iii)]{Benson/Lim:2025a} applies for the descent algebra. In particular, the Ext-quivers of $\Des_F(W)$ and $\Des_{\mathbb{C}}(W)$ are identical. 
\end{proof}

This result will later be used for Magma computation of the Ext-quiver of $\Des_F(W)$ for some specific Coxeter group $W$ but all primes $p$ satisfying $p\nmid n_W$. By which, we can just assume that $p=\infty$. 



\subsection{Type $\B$ descent algebra}\label{SS:typeB} 
We now focus on descent algebras of type $\B$. For each $n\geq 2$, we fix the following Coxeter graph.

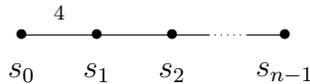
\[\begin{tikzpicture}
\node at (0,0) {$\sb$};
\node at (1,0) {$\sb$};
\node at (2,0) {$\sb$};
\node at (3.5,0) {$\sb$};
\node at (0,-.5) {$s_0$};
\node at (1,-.5) {$s_1$};
\node at (2,-.5) {$s_2$};
\node at (3.5,-.5) {$s_{n-1}$};
\node at (0.5,.3) {\tiny{$4$}};
\draw[-] (0,0) to (2.5,0);
\draw[-] (3,0) to (3.5,0);
\draw[-,dotted] (3,0) to (2.5,0);
\end{tikzpicture}\] The Coxeter group $W=\B_n$ is the hyperoctahedral group and we identify it with the subgroup $C_2\wr \sym{n}$ of the symmetric group $\sym{[-n,-1]\cup [1,n]}$ where $s_0=(-1,1)$ and $s_i=(i,i+1)(-i,-(i+1))$. We further identify the set of generators $\{s_0,\ldots,s_{n-1}\}$ with $[0,n-1]$ by specifying their indices.

We want to discuss a suitable labelling for the standard parabolic subgroups using compositions. For this we follow the account in    \cite{Atkinson/Pfeiffer/vanWilligenburg:2002a}.  For  $K\subseteq S$, the Coxeter element $c_K$ maybe used to label  the parabolic subgroup $W_K$ as follows. 

Since we have identified elements in $\B_n$ as permutations, we can write $c_K=x_0x_1$ where $x_0$ is the cycle containing both positive and negative elements, or $x_0=1$ if there is no such cycle, and
$x_1$ is the product of all the other cycles where positive and negative
cycles come in matching pairs. Notice that $x_0$ and $x_1$ commute. Each positive cycle is on some range $\{u,u+1,\ldots,v\}$ of consecutive integers, and the list of lengths of positive cycles taken in the natural order determines
and is determined by $K$. In this way, the subset $K$ of $S$ can be parametrised by compositions of integers $m$, with $0\leq m\leq n$.

Let $m_0 := n-m$. Then the  parabolic subgroup 
$W_K$ is conjugate to $(C_2\wr \sym{m_0}) \times \sym{q}\simeq \B_{m_0}\times \sym{q}$, where $\sym{q}$ is a direct product of symmetric groups (a Young subgroup), with factors corresponding to the parts of the composition  $q$ associated to $\sym{m}$. Notice that we only record the length, and we do not
need to know on which subsets of $[0, n-1]$ the group $\sym{q}$ acts. We write $q_K$ for $q$ the composition associated to $K$. The following from \cite{Atkinson/Pfeiffer/vanWilligenburg:2002a} is attributed to \cite{Bergeron/Bergeron/Howlett/Taylor:1992}.

\begin{lemma}\label{L:typeBcompo}  Let $K,L$ be subsets of $S$ and $q_K,q_L$ be the compositions associated to $K,L$ respectively.
\begin{enumerate}[(i)]
  \item Then $K$ is conjugate to $L$ if and only if $q_K$ and $q_L$ define the same partition.
  \item Suppose that $q_K$ has $a_i$ parts
	of size $i$ and $t$ is the length of $q_K$. We have
	\[[N(W_K):W_K]  =  2^t \prod_{i=1}^\infty a_i!.\]
\end{enumerate}	
\end{lemma}

Two subsets $K$ and $L$ of $S$ are conjugate in $\B_n$ if and only if $W_K$ and $W_L$ are conjugate in $W$, that is, the corresponding compositions give the same partition. With this, we can now specify the choices of $\R$ and $\R_p$ for $\B_n$. First we take for $\R$ the set consisting of subsets $J\subseteq S$ such that $q_J$ is a partition (of size $\leq n$). Moreover, using Lemma \ref{L:typeBcompo}(ii), $\R_p$ consists of the subsets $J\in\R$ such that $q_J$ is $p$-regular (that is, does not have $p$ or more equal parts). We write $\P(\leq n)$ (respectively, $\P_p(\leq n)$) for the set of all partitions (respectively, $p$-regular partitions) of sizes $\leq n$. This gives us the following.

\begin{proposition}\label{P:BSimple} The set $\P_p(\leq n)$ parametrises the complete set of non-isomorphic simple $\Des_F(\B_n)$-modules.
\end{proposition}

We give an example illustrating what we have discussed above. 

\begin{example} Consider $W=\B_2 = \langle s_0, s_1\rangle$. We list  the subsets $J$ of $S=\{s_0,s_1\}$ up to conjugation and their associated partitions.
\[\begin{array}{|c|c|c|}
\hline
J&q_J&[\N_W(W_J):W_J]\\ \hline
\emptyset  & (1^3) & |W|=8\\ \hline
\{s_1\} & (2,1) & 4\\ \hline
\{ s_1, s_2\} & (3) & 2 \\ \hline
\{ s_0\} & (1^2) & 2 \\ \hline
\{ s_0, s_2\} & (2) & 2 \\ \hline
\{ s_0, s_1\} & (1) & 2\\ \hline
S & \varnothing & 1 \\ \hline
\end{array}
\]
In the table of marks, taking this order, the diagonal entries are the $[N_W(W_J):W_J]$. The entries in the first column are the fixed points for the Coxeter element
for $J=\emptyset$, that is, the identity element. So they are the sizes of the coset spaces. The first row in the table of marks is $|W|, 0, 0, \ldots, 0$ (the only element which has fixed points on $W$ is the identity). The last column consists of zeros, except a $1$ in the last row. The last row has all entries equal to $1$. We see that, for $p=2$, all rows except the last one are zero modulo $2$. 
For $p=3$,  only the first row is zero.
\end{example}

In the classification of the representation type of the descent algebra of type $\A$, in \cite{Erdmann/Lim:2025}, the authors made use of surjective algebra homomorphisms connecting such algebras of different degrees (see \cite{Bergeron/Garsia/Reutenauer:1992a}). In the type $\B$ case, we will apply the following theorem by Aguiar--Bergeron--Nyman. In the theorem and its proof, $J+i$ is the set $\{j+i:j\in J\}$. 

\begin{theorem}[{\cite[Proposition 5.2]{Aguiar/Bergeron/Nyman:2004}}]\label{T:surjB} The map $\beta:\Des_F(\B_{n+1})\to \Des_F(\B_n)$ defined by, for each $K\subseteq [0,n]$,  \[\beta(x_K)=\left \{\begin{array}{ll} x_{(K\backslash \{0\})-1}&\text{if $0\in K$,}\\ 0&\text{if $0\not \in K$,}\end{array}\right .\] is a surjective algebra homomorphism.
\end{theorem}
\begin{proof} Let $B_J=x_{J^\c}$. In \cite{Aguiar/Bergeron/Nyman:2004}, it is shown that $\beta:\Des_\Z(\B_{n+1})\to \Des_\Z(\B_n)$ defined by \[\beta(B_J)=\left \{\begin{array}{ll} B_{J-1}&\text{if $0\not\in J$,}\\ 0&\text{if $0\in J$,}\end{array}\right .\] is a surjective algebra homomorphism over $\Z$ using the multiplication rule for the descent algebra of type $\B$ (see \cite{Bergeron/Bergeron:1992}). The statement is obtained by translating their result and then followed by taking modulo $p$.
\end{proof}

With this, we obtain the following corollary. 

\begin{corollary}\label{C:wildrep} If $\Des_F(\B_n)$ is infinite type (respectively, wild), then $\Des_F(\B_{n+1})$ is infinite type (respectively, wild).
\end{corollary}
\begin{proof} Use Theorem \ref{T:surjB} and Lemma \ref{L: surj alg}.
\end{proof}

\subsection{Type $\D$ descent algebra} In this subsection, we focus on the descent algebra of type $\D$. For each $n\geq 4$, we have the following Coxeter graph of type $\D_n$. 
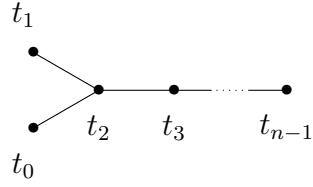
\[\begin{tikzpicture}
\node at (-.866,.5) {$\sb$};
\node at (-.866,-.5) {$\sb$};
\node at (0,0) {$\sb$};
\node at (1,0) {$\sb$};
\node at (2.5,0) {$\sb$};
\node at (0,-.5) {$\t_2$};
\node at (1,-.5) {$\t_3$};
\node at (-1,1) {$\t_1$};
\node at (-1,-1) {$\u$};
\node at (2.5,-.5) {$\t_{n-1}$};
\draw[-] (0,0) to (-.866,.5);
\draw[-] (0,0) to (-.866,-.5);
\draw[-] (0,0) to (1.5,0);
\draw[-] (2,0) to (2.5,0);
\draw[-,dotted] (2,0) to (1.5,0);
\end{tikzpicture}\]  \ Let $T=\{\t_i:i\in [0,n-1]\}$ the set of generators of $\D_n$. It is natural to take $\D_n$ as a subgroup
of $\B_n=\langle s_0,s_1,\ldots,s_{n-1}\rangle$ (as in the previous subsection)  generated by $\{\u, s_1, \ldots, s_{n-1}\}$ where $\u= s_0s_1s_0 =(-2,1)(-1,2)$ and $\t_i=s_i=(i,i+1)(-i,-(i+1))$. We  identify again the generator $\t_i$ by its index $i$ for $i\in [0,n-1]$. 

We summarise the description of parabolic subgroups following \cite{Atkinson/Pfeiffer/vanWilligenburg:2002a}. Let $K\subseteq T$. Then the parabolic subgroup $W_K$ is conjugate to $W_0\times W_1$ where $W_0\simeq \D_{n_0}$ (which may be  identity). By convention, we take $\D_0=1$, $\D_1\simeq \A_1$, $\D_2\simeq \A_1\times \A_1$ and $\D_3\simeq \A_3$. We consider two cases when $W_0\neq 1$ or $W_0=1$ below. 
\begin{enumerate}[(i)]
  \item Assume $W_0\neq 1$. Then  $W_0\cong  \D_{n_0}$ for some $n_0\geq 1$  and $W_1$ is a subgroup  of  $\langle  \t_{n_0+2}, \ldots, \t_{n-1} \rangle \cong \sym{m}$ for $m=n-n_0-1$. The smallest cases are $n_0=1,2,3$. When $n_0=1$, $W_0\simeq \D_1$ is either generated by $\u$ or $\t_1$. When $n_0=2$, we have $W_0= \langle \u, \t_1\rangle \simeq \D_2$. When  $n_0=3$, we have that  $W_0 = \langle \u, \t_1, \t_2\rangle \simeq \D_3$. Hence $W_K$ is determined by $W_1$ which is a standard parabolic subgroup of $\sym{m}$ and we can label the group $W_K$ by a composition of $m$.
  \item Assume $W_0=1$. Then $W_1$ is a subgroup of $\D_n$ either generated by $T'=\{ \t_1, \t_2, \ldots, \t_{n-1}\}$ or $T''=\{\u, \t_2, \ldots, \t_{n-1}\}$, which both are isomorphic to $\sym{n}$. In this case, we can label $W_K$ by a composition  $\mu$ of $n$.  If $\mu_1>1$, then $K$ contains either $\t_1$ or $\u$. To distinguish these two cases, we write  $(\mu, +)$  when $K$ contains $\t_1$, and we write $(\mu, -)$ when  $K$  contains $\u$.
\end{enumerate}  In any of the cases above, we write $\mu_K$ for the composition associated to the subset $K$ of $T$.


\begin{lemma}[{\cite[Lemma 10]{Atkinson/Pfeiffer/vanWilligenburg:2002a}}]\label{L: typeD Normaliser} Let $K\subseteq T$ and $\mu_K$ be the composition associated to $K$. Suppose that $\mu_K$ has $m_i$ parts of size $i$. Then 
\[[\N_{W}(W_K):W_K]=a\prod_{i=1}^\infty m_i!2^{m_i}\] where $a=\frac{1}{2}$ if $|\mu_K|=n$ and $\mu_K$ has an odd part; otherwise, $a=1$.
\end{lemma}

We see that partitions of $n$ such that all parts are even play a special role, they are called all-even.  Two subsets $K$ and $L$  of $S$ are conjugate in $W$ if and only if the corresponding compositions determine the same partition, unless they are all-even. In  that case, $K$ and $L$ are conjugate  if only if $K\cap L$ contains exactly one of the two generators $\u,\t_1$.  These show that we can take $\R$  to be the set consisting of subsets $K\subseteq T$ such that the associated compositions $\mu_K$ belongs to
\[\Gamma(n) := \P(\leq n-2) \cup \{ \lambda \vdash n : \text{$ \lambda_j$ is odd for some $j$}\} \cup \{ (\lambda, \pm ) : \text{$\lambda \vdash n$ and $\lambda$ is all-even} \}.\]
Then $\R_p$ is the subset of  $\R$ consisting of $K$ such that $\mu_K$ satisfies $p\nmid a\prod_{i=1}^\infty m_i!2^{m_i}$. Let $\Gamma_p(n)$ be the corresponding subset of $\Gamma(n)$. When $p=2$ and $n$ is even, $\Gamma_p(n)=\{\varnothing\}$. When $p=2$ and $n$ is odd, $\Gamma_p(n)=\{\varnothing,(n)\}$. When $p\neq 2$, $\Gamma_p(n)$ consists of the $p$-regular partitions belonging in $\Gamma(n)$.

\begin{proposition} The set $\Gamma_p(n)$ parametrises the complete set of non-isomorphic simple $\Des_F(\D_n)$-modules.
\end{proposition}


We end this section with another result by Aguiar--Bergeron--Nyman connecting descent algebras of type $\D$ and $\B$.

\begin{theorem}[{\cite[Theorem 5.12]{Aguiar/Bergeron/Nyman:2004}}]\label{T:surjD} Suppose that $n\geq 4$. There is a surjective algebra homomorphism $\gamma:\Des_F(\D_n)\to \Des_F(\B_{n-2})$ given by \[\gamma(x_J)=\left \{\begin{array}{ll}
x_{J-2}&\text{if $0, 1 \not\in J$,}\\
0&\text{otherwise,}\end{array}
\right .\] where $J-2=\{j-2:j\in J\}$.
\end{theorem}
\begin{proof} As mentioned in \cite{Aguiar/Bergeron/Nyman:2004}, the result is obtained using the multiplication rules for types $\B$ and $\D$ in \cite{Bergeron/Bergeron:1992,Bergeron/vanWilligenburg:1998}. Since the coefficients are integers, taking modulo $p$, we obtain the statement.
\end{proof}

Similar as Corollary \ref{C:wildrep}, we have the following:

\begin{corollary}\label{C:wildrepD} Let $n\geq 4$. If $\Des_F(\B_{n-2})$ is infinite type (respectively, wild), then $\Des_F(\D_n)$ is infinite type (respectively, wild).
\end{corollary}
\begin{proof} Use Theorem \ref{T:surjD} and Lemma \ref{L: surj alg}.
\end{proof}

\section{The Ext-quivers of the Descent Algebras of Types  $\B$ and $\D$}\label{S:ExtBandD}

In this section, we examine the Ext-quivers of the descent algebras of types $\B$ and $\D$. These are essential for our determination of their representation types in the subsequent sections. In particular, we focus on the cases for type $\B_n$ for small $n$ and independently $p=2$. 

\subsection{The Ext-quivers for Type $\B$} We first begin with the following corollary. In the previous section, we have established surjective algebra maps between descent algebras of type $\B$. Using Lemma \ref{L:surjExt} and Theorem \ref{T:surjB}, we get the following:

\begin{corollary}\label{C:subquiver} Let $Q_n$ and $Q_{n+1}$ be the Ext-quivers of $\Des_F(\B_n)$ and $\Des_F(\B_{n+1})$ respectively. Then $Q_n$ is a subquiver of $Q_{n+1}$.
\end{corollary}

In the case when $p\nmid |\B_n|$, we have the following explicit description of the Ext-quiver of $\Des_F(\B_n)$ due to Saliola.

\begin{theorem}[{\cite[Theorem 9.1]{Saliola:2008a}}]\label{T:SaliolaB} Suppose that $p\nmid |\B_n|$ (in this case, $\P_p(\leq n)=\P(\leq n)$). The Ext-quiver $Q$ of $\Des_F(\B_n)$ has vertices labelled by $\P(\leq n)$ and, for any $\lambda,\mu\in\P(\leq n)$, the number of arrows $\am_{\lambda,\mu}$ from $\lambda$ to $\mu$ is
\[\am_{\lambda,\mu}=\left \{\begin{array}{ll}
2&\text{if there exist $a>b>c\geq 1$ and $\delta$ such that}\\
&\text{$\lambda\approx \delta\cont (a+b+c)$ and $\mu\approx \delta\cont (a,b,c)$,}\\
1&\text{if there exist $a,b\geq 1$, $a\neq b$ and $\delta$ such that}\\
&\text{$\lambda\approx \delta\cont (a+2b)$ and $\mu\approx \delta\cont (a,b,b)$,}\\
1&\text{if there exist $a>b\geq 1$ such that $\mu\approx \lambda\cont (a,b)$,}\\
0&\text{otherwise.}
\end{array}
\right .\]
\end{theorem}

The computation of the Ext-quiver of $\Des_\F(\B_n)$ remains open when $p\mid |\B_n|$. However, we have the following special case. 

\begin{lemma}\label{L:p2ExtQ} Let $p=2$. The Ext-quiver of $\Des_F(\B_n)$ is a single vertex with at least $n$ loops.
\end{lemma}
\begin{proof} The algebra $A=\Des_F(\B_n)$ has a unique simple module and $\Rad(A)$ has the basis $\{x_J:J\neq [0,n-1]\}$. For each $i\in [0,n]$, let \[X(i)=\Span\{x_J:|J|\leq n-i\}.\] We claim that $\Rad^m(A)\subseteq X(m)$.

It is true for $m=0,1$. Suppose that $\Rad^m(A)\subseteq X(m)$ for some $m\in [1,n]$. We consider $\Rad^{m+1}(A)=\Rad(A)\Rad^m(A)$. Let $|K|\leq n-m$ and $x_J\in \Rad(A)$, i.e., $J\neq [0,n-1]$. By Lemma \ref{L:beta}, we have $2\mid a_{JJJ}\mid a_{JKK}$. Therefore, $x_Jx_K\in X(m+1)$. As such, $\Rad^{m+1}(A)\subseteq \Rad(A)X(m)\subseteq X(m+1)$. Apply to the case $m=1$, we have $\dim_F(\Rad(A)/\Rad^2(A))\geq \dim_F(X(1)/X(2))=n$. Thus the Ext-quiver of $A$ has at least $n$ loops.
\end{proof}

We end this subsection with the following conjecture.

\begin{conjecture} Let $p=2$ and $M$ be the unique simple module for the descent algebra $\Des_F(\B_n)$. The
multiplicity of $M$ in the $i$th radical layer of the projective cover $\Proj(M) = \Des_F(\B_n)$ is given by the
binomial coefficient ${n\choose i}$. In particular, its Ext-quiver is a single vertex with $n$ loops:
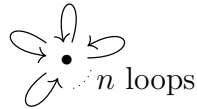
\[\begin{tikzcd}
\sb \arrow[out=0,in=30,loop,distance=20pt]\arrow[out=72,in=102,loop,distance=20pt] \arrow[out=144,in=174,loop,distance=20pt]\arrow[out=216,in=246,loop,distance=20pt]
\end{tikzcd}
\begin{tikzpicture}[remember picture, overlay]
\draw [dotted] (-.9,-.3) to[bend right] (-.65,0);
\node at (0,-.2) {\ $n$ loops};
\end{tikzpicture}\]
\end{conjecture}

\begin{remark}\label{R:p2ExtQ} To prove the conjecture, we are left to prove that $X(m+1)\subseteq \Rad^{m+1}(A)$ as in the proof of Lemma \ref{L:p2ExtQ}. 
\end{remark}

\subsection{Type $\B$ descent algebras for small $n$}  We now focus  on the Coxeter groups $\B_n$ where $2\leq n\leq 6$. In particular, we aim to describe not only their Ext-quivers but structures which will be useful for us in the subsequent sections.

Recall the simple modules for descent algebras given in Lemma \ref{L: matrix M_F}. To begin, for each $\lambda\in \P_p(\leq n)$, let $J\subseteq [0,n-1]$  be the corresponding subset of $S$, and we write \[M_\lambda:=M_{J,F}\] by suppressing the field $F$. For example, $M_\varnothing=M_{[0,n-1],F}$. 

In order to study the representation type of $\Des_F(\B_n)$, we will need to understand the algebra in detail for some small values of $n$. They are presented in Lemmas \ref{L:p3n4iso}, \ref{L:Q22}, \ref{L:n5} and \ref{L:n6} below.

\begin{lemma}\label{L:p3n4iso}\
\begin{enumerate}[(i)]
\item Suppose that $p\geq 5$. We have $\Des_F(\B_4)\cong F Q$ where $Q$ is the Ext-quiver of $\Des_F(\B_4)$ given below:  \[\begin{tikzcd}[column sep=1.5em]
31&\varnothing\arrow[l]\arrow[r]& 21& 4\arrow[r]&211&1\arrow[l]&
3&2^2&2&1^2&1^3 & 1^4
\end{tikzcd}
\]
\item Suppose that $p=3$. The Ext-quiver $Q$ of $\Des_F(\B_4)$ is given as follows:
\[\begin{tikzcd}[column sep=1.5em]
31\arrow[out=75,in=105,loop,distance=20pt,swap,"\varepsilon"]&\varnothing\arrow[l,"\alpha"]\arrow[r]& 21& 4\arrow[r]&211&1\arrow[l]&
3\arrow[out=75,in=105,loop,distance=20pt,swap,"\delta"]& 2^2&2&1^2
\end{tikzcd}
\] Furthermore, $\Des_F(\B_4)\cong F Q/(\varepsilon^2,\delta^2,\varepsilon\alpha)$.
\end{enumerate}
\end{lemma}
\begin{proof} (i) Since $5\nmid |\B_4|$, by Theorem \ref{T:SaliolaB}, $Q$ is given as in the statement. Since $\dim_F F Q=16=\dim_F\Des_F(\B_4)$, we get the isomorphism.

(ii) We claim that there are loops at both $3$ and $31$ of $Q$. Once we have proved this, by Corollary \ref{C:subquiver} and part (i), $Q$ contains the subquiver as in the statement. However, the total number of vertices and arrows of the quiver above is $16$ which is the dimension of $\Des_F(\B_4)$. As such, $Q$ must be exactly as stated and $\Des_F(\B_4)\cong F Q/(\varepsilon^2,\delta^2,\varepsilon\alpha)$.

Fix the total order for $\P(\leq 4)$ (and hence for $\P_3(\leq 4)$ as a subset) as follows:
\[(1^4)<( 2, 1^2 )<(1^3)<
( 3, 1 )<
( 2, 2 )<
( 2, 1 )<
( 1^2 )<
( 4 )<
( 3 )<
( 2 )<
( 1 )<
\varnothing.\] By part (i), it is easy to write down the Cartan matrix $C$ of $\Des_\Q(\B_4)$. Let $D$ be the decomposition matrix. By Theorem \ref{T: APW}, the Cartan matrix for $\Des_{\Fp}(\B_4)$ is
\[\wt C=D^\top CD={\scriptsize\begin{bmatrix}
1&0&0&0&0&0&0&0&0&0\\
0&2&0&0&0&0&0&0&0&0\\
0&0&1&0&0&0&0&0&0&0\\
0&0&0&1&0&0&0&0&0&0\\
0&0&0&0&1&0&0&0&0&0\\
1&0&0&0&0&1&0&0&0&0\\
0&0&0&0&0&0&2&0&0&0\\
0&0&0&0&0&0&0&1&0&0\\
1&0&0&0&0&0&0&0&1&0\\
0&1&0&1&0&0&0&0&0&1
\end{bmatrix}}\] From the (second and seventh) rows of $\wt C$ labelled by $(3,1)$ and $(3)$, we observe that the respective projective covers are
\begin{align*}
\Proj(M_{31})&=\begin{bmatrix} M_{31}\\ M_{31}\end{bmatrix},\quad
\Proj(M_{3})=\begin{bmatrix} M_{3}\\ M_{3}\end{bmatrix}.\qedhere
\end{align*}
\end{proof}

\begin{lemma}\label{L:Q22} Let $p=2$.
\begin{enumerate}[(i)]
\item The Ext-quiver $Q$ of $\Des_F(\B_2)$ is $\begin{tikzcd}
\sb \arrow[out=165,in=195,loop,distance=20pt,swap,"\alpha"]\arrow[out=15,in=-15,loop,distance=20pt,"\beta"]\end{tikzcd}$ and $\Des_F(\B_2)\cong F Q/(\alpha^2,\beta^2-\alpha\beta,\alpha\beta-\beta\alpha)$.
\item The Ext-quiver of $\Des_F(\B_3)$ is $\begin{tikzcd}
\sb\arrow[out=165,in=195,loop,distance=20pt]\arrow[out=45,in=75,loop,distance=20pt]\arrow[in=-45,out=-75,loop,distance=20pt]\end{tikzcd}$.
\end{enumerate}
\end{lemma}
\begin{proof} (i) Let $A:=\Des_F(\B_2)$. Since $p=2$, there is a unique $\Des_F(\B_2)$-module $M$. Let $x=x_{\{0\}}$, $y=x_{\{1\}}$ and $z=x_\emptyset$. We have $x,y,z\in\Rad(A)$. It is easy to check that $x^2=z=y^2$ and $xy=0=yx$. We have the regular module \begin{equation}\label{Eq:Q22} \Proj(M)=\begin{bmatrix} M\\ M\oplus M\\ M\end{bmatrix}.
\end{equation} This explains the Ext-quiver $Q$. The isomorphism is obtained by identifying $\alpha,\beta$ with $x+y$ and $y$ respectively.

(ii) We have seen that $\Rad^2(A)\subseteq X(2)$ (see the notation in the proof of Lemma \ref{L:p2ExtQ}). In this case, $\Rad^2(A)=\Span\{x_J:|J|=0,1\}=X(2)$. So the Ext-quiver of $\Des_F(\B_3)$ is as stated.
\end{proof}

For the proof of the next lemma, we require the following general fact.

\begin{proposition}\label{P:projsimple} Let $A$ be an $\F$-algebra and $M$ be an $A$-module. Then every projective composition factor of $M$ is a summand of the socle of $M$.
\end{proposition}
\begin{proof} Let $N_1\subseteq N_2\subseteq M$ be submodules of $M$ such that $N_2/N_1\cong S$  where $S$ is simple and projective. Then $N_2=N_1\oplus S'$ where $S'\cong S$ is a submodule of $N_2$. So $S'$ belongs in the socle of $M$. 
\end{proof}

\begin{lemma}\label{L:n5}\
\begin{enumerate}[(i)]
\item Suppose that $p\geq 7$. The descent algebra $\Des_F(\B_5)$ is a path algebra isomorphic to $F Q$ where $Q$ is the Ext-quiver of $\Des_F(\B_5)$ given below: 
\[\begin{tikzcd}[row sep=1.2em, column sep=1.2em]
&32&&&&2^2&1^4\\
21&\varnothing\arrow[u]\arrow[l]\arrow[d]\arrow[r]& 41\arrow[r]&21^3&1^2\arrow[l]&3&1^3\\
&31&&&&1^5\\
2\arrow[r]&2^21&5\arrow[l]\arrow[r]& 31^2&1\arrow[l]\arrow[r]&21^2&4\arrow[l]
\end{tikzcd}
\]
\item Suppose that $p=5$. The Ext-quiver $Q$ of $\Des_F(\B_5)$ is given below:
\[\begin{tikzcd}[row sep=1.2em, column sep=1.2em]
&32&&&&2^2&1^4\\
21&\varnothing\arrow[u]\arrow[l]\arrow[d]\arrow[r]& 41\arrow[r]&21^3&1^2\arrow[l]&3&1^3\\
&31&\\
2\arrow[r]&2^21&5\arrow[l,swap,"\beta"]\arrow[r,"\alpha"] \arrow[out=75,in=105,loop,distance=20pt,swap,"\varepsilon"]& 31^2&1\arrow[l]\arrow[r]&21^2&4\arrow[l]
\end{tikzcd}
\] Furthermore, $\Des_F(\B_5)\cong F Q/(\varepsilon^2,\alpha\varepsilon,\beta\varepsilon)$.
\end{enumerate}
\end{lemma}
\begin{proof} (i) The Ext-quiver $Q$ is given by Theorem \ref{T:SaliolaB}. By direct calculation, the total numbers of vertices of $Q$ is $19$, arrows is $12$, paths of length two is 1, and there is no path of length more than two. Since $\dim \Des_F(\B_5)=32=19+12+1$, we have the desired isomorphism by Theorem \ref{T:Gabriel}.

(ii) By Theorem \ref{T:SaliolaB} and Lemma \ref{L:Extmodp}, we obtain all the arrows as in the statement except the loop at $(5)$. At this stage, the total number of vertices and arrows is $30$. We fix the following total order on $\P(\leq 5)$:
\begin{align*}
&(1^5)<( 2, 1^3 )<
( 1^4)<
( 3, 1^2 )<
( 2^2, 1 )<
( 2, 1^2 )<
( 1^3 )<
( 4, 1 )<
( 3, 2 )<\\
&( 3, 1 )<
( 2,^2 )<
( 2, 1 )<
( 1^2 )<
( 5 )<
( 4 )<
( 3 )<
( 2 )<
( 1 )<
\varnothing.
\end{align*} By part (i), since $\Des_\Q(\B_5)\cong \Q Q$, the Cartan matrix of $\Des_\Q(\B_5)$ is easy to write down. Let $D$ be the decomposition matrix. As such, by Theorem \ref{T: APW}, we obtain
\[\wt C=D^\top CD={\scriptsize\begin{bmatrix}
1&0&0&0&0&0&0&0&0&0&0&0&0&0&0&0&0&0\\
0&1&0&0&0&0&0&0&0&0&0&0&0&0&0&0&0&0\\
0&0&1&0&0&0&0&0&0&0&0&0&0&0&0&0&0&0\\
0&0&0&1&0&0&0&0&0&0&0&0&0&0&0&0&0&0\\
0&0&0&0&1&0&0&0&0&0&0&0&0&0&0&0&0&0\\
0&0&0&0&0&1&0&0&0&0&0&0&0&0&0&0&0&0\\
1&0&0&0&0&0&1&0&0&0&0&0&0&0&0&0&0&0\\
0&0&0&0&0&0&0&1&0&0&0&0&0&0&0&0&0&0\\
0&0&0&0&0&0&0&0&1&0&0&0&0&0&0&0&0&0\\
0&0&0&0&0&0&0&0&0&1&0&0&0&0&0&0&0&0\\
0&0&0&0&0&0&0&0&0&0&1&0&0&0&0&0&0&0\\
1&0&0&0&0&0&0&0&0&0&0&1&0&0&0&0&0&0\\
0&0&1&1&0&0&0&0&0&0&0&0&2&0&0&0&0&0\\
0&0&0&0&1&0&0&0&0&0&0&0&0&1&0&0&0&0\\
0&0&0&0&0&0&0&0&0&0&0&0&0&0&1&0&0&0\\
0&0&0&1&0&0&0&0&0&0&0&0&0&0&0&1&0&0\\
0&0&1&0&1&0&0&0&0&0&0&0&0&0&0&0&1&0\\
1&0&0&0&0&0&1&1&1&0&1&0&0&0&0&0&0&1
\end{bmatrix}}\] The matrix $\wt C$ is obtained from $C$ by removing both the first row and column and then changing the entry corresponding to $((5),(5))$ from 1 to 2. By Theorem \ref{T:Gabriel}, $A:=\Des_F(\B_5)\cong F Q/I$ for some admissible ideal $I$ of $F Q$. Let $\Proj_\lambda=\Proj(M_\lambda)$ the projective cover of the simple $A$-module $M_\lambda$. Examining $\Proj_{\varnothing}$, we see that the path $\gamma:\varnothing\to 41\to 21^3$ of length 2 does not belong in $I$. If we can show that there is a loop at $(5)$, then the total number of vertices, arrows and loops, and together with the path $\gamma$ is 32, which is the dimension of $A$. If so, the proof is complete.

The Cartan matrix $\wt C$ shows that  $W:=\Rad(\Proj_{5})$ has composition factors $M_{5}$, $M_{31^2}$ and $M_{2^21}$ where both $M_{31^2}$ and $M_{2^21}$ are projective. By Proposition \ref{P:projsimple}, $M$ has a submodule $T\cong M_{31^2}\oplus M_{2^21}$ and hence $W/T \cong M_{5}$. Thus $M_{5}$ belongs to $W/\Rad(W)$ and hence there is a loop at the vertex $(5)$.
\end{proof}

\begin{lemma}\label{L:n6} Suppose that $p\geq 5$. The Ext-quiver of $\Des_F(\B_6)$ contains the following subquiver:
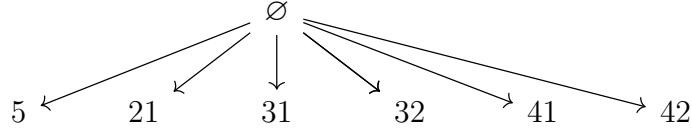
\[\begin{tikzcd}
&&\varnothing\arrow[dll]\arrow[dl]\arrow[d]\arrow[dr]\arrow[dr]\arrow[drr]\arrow[drrr]\\
5&21&31&32&41&42
\end{tikzcd}
\]
\end{lemma}
\begin{proof} The Ext-quiver $Q$ of $\Des_\Q(\B_6)$ contains such subquiver in the statement by Theorem \ref{T:SaliolaB}. We now obtain the $p>0$ case using Lemma \ref{L:Extmodp}.
\end{proof}

\subsection{The Ext-quivers for Type $\D$} We now study the Ext-quiver of the descent algebra of type $\D$ over fields of characteristic $2$.  

\begin{lemma}\label{L:Qp2neven} Suppose that $p=2$ and $n$ is even. The Ext-quiver of $\Des_F(\D_n)$ is a single vertex with at least $n$ loops.
\end{lemma}
\begin{proof} Let $A:=\Des_F(\D_n)$ and $T=[0,n-1]$. Since $n$ is even, there is a single simple $A$-module. Furthermore, $\Rad(A)$ has a basis $\{x_J:J\subsetneq T\}$. We claim that $\Rad^2(A)$ is contained  in the subspace $U$ spanned by $x_L$'s where $|L|\leq n-2$. Let $x_J,x_K\in \Rad(A)$, i.e., $J,K\neq T$. By Lemma \ref{L:beta}, we have $2\mid a_{JJJ}\mid a_{JKK}$. Therefore, $x_Jx_K$ must be a linear combination of $x_L$'s such that $L\subsetneq K$. This justifies our claim. As such, $\dim_F\Rad(A)/\Rad^2(A)\geq \dim_F\Rad(A)/U=n$.
\end{proof}

\begin{example} Suppose that $p=2$. The Ext-quiver of $\Des_F(\D_4)$ is
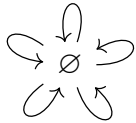
\[\begin{tikzcd}
\varnothing \arrow[out=0,in=30,loop,distance=20pt]\arrow[out=72,in=102,loop,distance=20pt] \arrow[out=144,in=174,loop,distance=20pt]\arrow[out=216,in=246,loop,distance=20pt]
\arrow[out=288,in=318,loop,distance=20pt]
\end{tikzcd}\]
\end{example}

\section{The Representation Type of the Descent Algebra of Type $\B$}\label{S:RepTypeB}

In this section, we clasify the representation type of the descent algebras of type $\B$ as in Theorem \ref{T:RepTypeB} below. Our strategy  involves the surjective algebra homomorphism stated in Theorem \ref{T:surjB}, Lemma \ref{L: surj alg} and examining the algebras $\Des_F(\B_n)$ for some small $n$.

\begin{theorem}\label{T:RepTypeB} Let $n\geq 2$. The algebra $\Des_F(\B_n)$ is
\begin{enumerate}[(i)]
\item finite type if $p\geq 3$ and $n\leq 4$,
\item tame if $p=2=n$,
\item wild otherwise.
\end{enumerate}
\end{theorem}

To prove Theorem \ref{T:RepTypeB}, we require the series of lemmas below. Basically, for each $p$, we need to find integers $n_0\leq n_1$ such that $\Des_F(\B_n)$ has respectively finite, tame, and wild representation types  when $n=n_0-1$, $n\in\{n_0,n_1-1\}$, and $n=n_1$.

\begin{lemma}\label{L:p2n2rep} Suppose that $p=2$. The algebra $\Des_F(\B_2)$ is tame.
\end{lemma}

\begin{proof} Let $A=\Des_F(\B_2)$. By Lemma \ref{L:Q22}(i), let $M$ be the unique simple $A$-module. Then $\mathsf{I}(M)\cong \Proj(M)$ (see Equation (\ref{Eq:Q22})). By Lemma \ref{L:projinj}, $B:=A/M$ has the same representation type as $A$. Notice that $\Rad^2(B)=0$ and the Ext-quiver $Q$ of $B$ is identical to that of $A$, that is, by Lemma \ref{L:Q22}(i), a single vertex with exactly two loops. Since the separated quiver of $Q$ is the Kronecker quiver \[\begin{tikzcd}
\sb\arrow[r,shift right=.7ex]\arrow[r,shift left=.7ex]&\sb\end{tikzcd}\] by Theorem \ref{T:separatedquiver}, $B$ is tame. Therefore, $A$ is tame.
\end{proof}

\begin{lemma}\label{L:p2n3rep} Suppose that $p=2$ and $n\geq 3$. The algebra $\Des_F(\B_n)$ is wild.
\end{lemma}
\begin{proof} By Corollary \ref{C:wildrep}, it suffices to prove for the case where $n=3$. By Lemma \ref{L:Q22}(ii), the algebra $\Des_\F(\B_3)$ satisfies Corollary \ref{C:wildquotient} and hence it is wild. 
\end{proof}

\begin{lemma}\label{L:p3n4rep} Suppose that $p\geq 3$ and $n\leq 4$. The algebra $\Des_F(\B_n)$ has finite type.
\end{lemma}
\begin{proof} By Corollary \ref{C:wildrep}, it suffices to prove for the case where $n=4$. For this case, we make use of Lemma \ref{L:p3n4iso}. From the lemma, regardless of whether  $p=3$ or $p\geq 5$, we have $\Rad^2(\Des_F(\B_4))=0$ and its separated quiver is a disjoint union Dynkin diagrams (of type $\A$). As such, $\Des_F(\B_4)$ has finite type by Theorem \ref{T:separatedquiver}(i) for all odd primes $p$.
\end{proof}

\begin{lemma}\label{L:p>2n5repB} Let $p\geq 3$ and $n\geq 5$. Then $\Des_F(\B_n)$ is wild.
\end{lemma}
\begin{proof} By Corollary \ref{C:wildrep}, it suffices to prove this  for the case where $n=5$. Let $A=\Des_F(\B_5)$. Suppose first that $p=3$. We have $M_{32,\Z}\otimes_\Z F\cong M_{21^3,\Z}\otimes_\Z F$. Using Lemma \ref{L:Extmodp}, Proposition 3.7 and Lemma \ref{L:n5}(i), we conclude that the Ext-quiver $Q$ of $A$ has a subquiver as follows.
\[\begin{tikzcd}[row sep=1.2em, column sep=1.2em]
&32&\\
21&\varnothing\arrow[u]\arrow[l]\arrow[d]\arrow[r]& 41\arrow[ul]\\
&31
\end{tikzcd}\] Therefore, $A$ satisfies Corollary \ref{C:wildquotient} and hence it is wild.

Next, suppose that $p=5$.  The Ext-quiver $Q$ of $A$ is  given in Lemma \ref{L:n5}(ii). Let $B$ be the block of $A$ corresponding  to the component $Q'$ of the $Q$ which contains the simple module labelled by $\varnothing$. The presentation shows that 
we have $\dim B=14$ and $B\cong F Q'$. By Theorems \ref{T: Gab} and \ref{T:tame}, $B$ is wild. Therefore, $A$ is wild.

Finally, suppose that $p\geq 7$. By Theorem \ref{T:SaliolaB} and Lemma \ref{L:n5}(i), $A\cong F Q$ where $Q$ is the Ext-quiver of $A$.  The direct summand $FQ'$ of $FQ$ is wild by Theorems \ref{T: Gab} and \ref{T:tame} where $Q'$ is the component of $Q$ containing the vertex $\varnothing$. So $A$ is wild. 
\end{proof}

Putting everything together, we obtain the classification of the representation type of the descent algebra of type $\B$.

\begin{proof}[Proof of Theorem \ref{T:RepTypeB}] Combine  Lemmas \ref{L:p2n2rep}, \ref{L:p2n3rep}, \ref{L:p3n4rep} and \ref{L:p>2n5repB}.
\end{proof}

\section{The Representation Type of the Descent Algebra of type $\D$}\label{S:RepTypeD}

In this section, we classify the representation type of the descent algebras of type $\D$ as given in Theorem \ref{T:RepTypeD} below.


\begin{theorem}\label{T:RepTypeD} Let $n\geq 4$. The descent algebra $\Des_F(\D_n)$ is
\begin{enumerate}[(i)]
  \item tame if $p\geq 5$ and $n=4$,
  \item wild otherwise.
\end{enumerate}
\end{theorem}

Unlike the type $\B$ case, there is no  surjective algebra homomorphism from a descent algebra of type $\D$ to another of the same type. In fact, by studying the Ext-quivers of $\Des_F(\D_4)$ and $\Des_F(\D_5)$ for $p\geq 7$ (see Figures \ref{F:TypeDn4p0} and \ref{F:TypeDn5p0} in Appendix \ref{Appendix:ExtD}), we conclude, using Lemma \ref{L:surjExt}, that $\Des_F(\D_5)$ does not map onto $\Des_F(\D_4)$ because, for example, we have two arrows from a vertex to another in the Ext-quiver of $\Des_F(\D_4)$ but there is no such subquiver in the Ext-quiver of $\Des_F(\D_5)$. 

Instead, to prove Theorem \ref{T:RepTypeD}, we make use of Corollary \ref{C:wildrepD} and our result for  type $\B$ in Theorem \ref{T:RepTypeB} to reduce our problem in type $\D_n$ to small cases of $n$. This yields the following lemma.

\begin{lemma}\label{L:wildDB} The descent algebra $\Des_F(\D_n)$ is wild if either
\begin{enumerate}[(i)]
  \item $n\geq 7$ and $p\geq 3$, or
  \item $p=2$.
\end{enumerate}
\end{lemma}
\begin{proof} For part (i), we use Lemma \ref{L: surj alg} and Theorems \ref{T:surjB} and \ref{T:RepTypeB}. For part (ii), by Lemma \ref{L:Qp2neven}, the Ext-quiver of $\Des_F(\D_n)$ contains a single vertex with at least $n$ loops where $n\geq 4$. By Corollary \ref{C:wildquotient}, $\Des_F(\D_n)$ is wild. 
\end{proof}

In view of Lemma \ref{L:wildDB}, to complete the classification, we are left with the cases $n=4,5,6$ when $p\geq 3$. In order to do this, we require Magma \cite{Magma} computation for their Ext-quivers. We briefly explain the feasibility of the computation.

In \cite[Section 8]{Benson/Lim:2025a}, Benson and the second author gave a construction for a complete set of primitive orthogonal idempotents of $\Des_F(W)$ for arbitrary Coxeter group $W$ and field $F$. Since \[\dim\Ext^1_A(S,T)=\dim (e_T\Rad(A)e_S/e_T\Rad^2(A)e_S),\] theoretically, we can compute the Ext-quiver of $\Des_F(W)$.  Furthermore, since we are only interested in the cases when $n=4,5,6$, by Corollary \ref{C:BensonLimExtQuiver}, there are only finite cases of prime numbers $p$ we need to consider, i.e., those $p$ such that $p=\infty$ or $p\mid n_{\D_n}$. We summarise in Table \ref{Table:n_D} below. For completeness, we include the numbers $d_{\D_n,i}$'s in \cite[Definition 9.9]{Benson/Lim:2025a}.

\begin{table}[H]
\centering
\renewcommand{\arraystretch}{1.2}
\begin{tabular}{|c|c|c|c|c|c|}
\hline
$n$ & $d_{\D_n,1}$ & $d_{\D_n,2}$ & $d_{\D_n,3}$ & $|\D_n|$ & $n_{\D_n}$ \\ \hline
4 & 1 & - & - & $2^6\cdot 3$ & $2^6\cdot 3$ \\ \hline
5 & 1 & $2^8\cdot 3^2$ & $2^3\cdot 3$ & $2^7\cdot 3\cdot 5$ & $2^8\cdot 3^2\cdot 5$ \\ \hline
6 & 1 & $2^{22}\cdot 3^5$ & - & $2^9\cdot 3^2\cdot 5$ & $2^{22}\cdot 3^5\cdot 5$ \\ \hline
\end{tabular}
\caption{Values of $d_{\D_n,i}$'s and $n_{\D_n}$ for $n=4,5,6$}\label{Table:n_D}
\end{table}

Notice that, from the table, the only primes $p$ dividing $n_{\D_n}$ are precisely the primes $p$ dividing $|\D_n|$ for these values of $n$ (see \cite[Question 9.13]{Benson/Lim:2025a}). The Ext-quivers of these algebras are presented in Appendix \ref{Appendix:ExtD} and we refer to the diagrams there for the proofs of the lemmas below.

\begin{lemma}\label{L:DExt2} Let $p\geq 5$ and $A=\Des_F(\D_4)$.
\begin{enumerate}[(i)]
\item The Jacobson radical $\Rad(A)$ is 5-dimensional with a basis \[\{x_{\{0\}}-x_{\{1\}},x_{\{1\}}-x_{\{2\}},x_{\{2\}}-x_{\{3\}}, x_{\{0,2\}}-x_{\{1,2\}},x_{\{1,2\}}-x_{\{2,3\}}\}.\]
\item The algebra $A$ is two-nilpotent.
\item Let $Q$ be the Ext-quiver of $A$. Then $A\cong F Q$.
\item The algebra $A$ has tame representation type.
\end{enumerate}
\end{lemma}
\begin{proof} From the general theory, we obtain the basis $B$ for $\Rad(A)$. It is straightforward to check using the multiplication rule for $A$ that $bb'=0$ for any $b,b'\in B$. Therefore, $A$ is 2-nilpotent. 
Let $Q$ be the Ext-quiver of $A$ as given in Figure \ref{F:TypeDn4p0}. By Theorem \ref{T:Gabriel}, $A\cong F Q/I$ for some admissible ideal $I$ of $F Q$. Examining $Q$, we observe that it does not contain any path of length two. So we must have $I=0$. Since $Q$ is a disjoint union of Dynkin diagrams of types $\A_1$ and $\D_4$, and the extended Dynkin diagram $\wt\A_2$, by Theorem \ref{T:tame}, $A$ is tame.
\end{proof}

\begin{lemma}\label{L:DtheRest} The algebra $\Des_F(\D_n)$ is wild for the following two cases:
\begin{enumerate}[(i)]
  \item $p=3$ and $n\in\{4,5,6\}$;
  \item $p\geq 5$ and $n\in\{5,6\}$.
\end{enumerate}
\end{lemma}
\begin{proof} Using Corollary \ref{C:wildquotient} and studying the Figures \ref{F:TypeDn4p3}, \ref{F:TypeDn5p3}, \ref{F:TypeDn5p5}, \ref{F:TypeDn6p3}, \ref{F:TypeDn6p5} and \ref{F:TypeDn6p0}, we obtain all cases in the statement with the exception of $n=5$ and $p\geq 7$. In this case, let $Q$ be the Ext-quiver of $\Des_F(\D_5)$ (see Figure \ref{F:TypeDn5p0}). By  comparing dimensions  we see that $\Des_F(\D_5)\cong F Q$. By Theorem \ref{T: Gab}, $\Des_F(\D_5)$ is wild.
\end{proof}

Putting everything together, we obtain the classification of the representation type of the descent algebra of type $\D$.

\begin{proof}[Proof of Theorem \ref{T:RepTypeD}] Combine Lemmas \ref{L:wildDB}, \ref{L:DExt2} and \ref{L:DtheRest}.
\end{proof}

\section{The Representation Type of the Descent Algebra of type $\I$}\label{S:RepTypeI}

In this section, we classify the representation type of the descent algebras of type $\I$. Let $n\geq 4$. Consider the Coxeter graph of $\I_n$ given below:
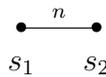
\[\begin{tikzpicture}
\node at (0,0) {$\sb$};
\node at (1,0) {$\sb$};
\node at (0,-.5) {$s_{1}$};
\node at (1,-.5) {$s_{2}$};
\draw[-] (0,0) to (1,0);
\node at (0.5,.2) {\tiny{$n$}};
\end{tikzpicture}\] The descent algebra $\Des_F(\I_n)$ has 4-dimensional with a basis $\{1,x_1,x_2,x_3\}$ where 
\begin{align*}
  1&=x_{\{s_1,s_2\}},& x_1&=x_{\{s_1\}},& x_2&=x_{\{s_2\}},& x_3&=x_\varnothing.
\end{align*} It is easy to work out the following lemma.

\begin{lemma}\label{L:TypeITable} Let $\ell=\lfloor \frac{n}{2}\rfloor$ and $\varepsilon=\delta_{n,\text{even}}$. We have the following multiplication table for $\Des_F(\I_n)$:
\[\begin{array}{c|cccc}
&1&x_1&x_2&x_3\\ \hline
1&1&x_1&x_2&x_3\\
x_1&x_1&(1+\varepsilon)x_1+(\ell-\varepsilon)x_3&(1-\varepsilon)x_2+\ell x_3&nx_3\\
x_2&x_2&(1-\varepsilon)x_1+\ell x_3&(1+\varepsilon)x_2+(\ell-\varepsilon)x_3&n x_3\\
x_3&x_3&nx_3&nx_3&2nx_3
\end{array}\]
\end{lemma}

With Lemma \ref{L:TypeITable}, we have the detailed structure for $\Des_F(\I_n)$ as follows. 

\begin{lemma}\label{L:TypeIsimple} Let $A=\Des_F(\I_n)$, $\ell=\lfloor \frac{n}{2}\rfloor$ and $Q$ be the Ext-quiver of $A$.
\begin{enumerate}[(i)]
  \item If $p=2$ and $n$ is odd, then there are two non-isomorphic simple $A$-modules $U$ and $V$. The projective indecomposable $A$-modules are $\Proj(U)=\begin{bmatrix} U\\ U\end{bmatrix}$ and $\Proj(V)=\begin{bmatrix} V\\ U\end{bmatrix}$. Furthermore, $Q$ is
  \[\begin{tikzcd}
V\arrow[r,"\alpha"]&U\arrow[out=15,in=-15,loop,distance=20pt,"\epsilon"]
\end{tikzcd}\] and $A\cong F Q/(\epsilon^2,\epsilon\alpha)$.
\item If $p=2$ and $n$ is even, then there is a single simple $A$-module $U$ and $\Proj(U)=\begin{bmatrix} U\\ U\oplus U\\ U\end{bmatrix}$. Furthermore, $Q$ is
\[\begin{tikzcd}
U\arrow[out=15,in=-15,loop,distance=20pt,"\alpha"] \arrow[out=195,in=165,loop,distance=20pt,"\beta"]
\end{tikzcd}\] and $A\cong F Q/(\alpha\beta-\beta\alpha,\alpha^2,\beta^2-(\ell+1)\alpha\beta)$.
\item If $p\geq 3$, $p\mid n$, and $n$ is odd, then there are two non-isomorphic simple $A$-modules $U$ and $V$ where $U$ is projective and $\Proj(V)=\begin{bmatrix} V\\ U\oplus V\end{bmatrix}$. Furthermore, $Q$ is
\[\begin{tikzcd}
U&V\arrow[l,"\alpha"]\arrow[out=15,in=-15,loop,distance=20pt,"\epsilon"]
\end{tikzcd}\] and $A\cong F Q/(\epsilon^2,\alpha\epsilon)$.
\item If $p\geq 3$, $p\mid n$, and $n$ is even, then there are three non-isomorphic simple $A$-modules $U$, $V$, and $Z$ where $U,V$ are projective and $\Proj(Z)=\begin{bmatrix} Z\\ Z\end{bmatrix}$.  Furthermore, $Q$ is
\[\begin{tikzcd}
U&V&Z\arrow[out=15,in=-15,loop,distance=20pt,"\epsilon"]
\end{tikzcd}\] and $A\cong F Q/(\epsilon^2)$.
\item If $p\geq 3$, $p\nmid n$, and $n$ is odd, then there are three non-isomorphic simple $A$-modules $U$, $V$ and $Z$ where $U,V$ are projective and $\Proj(Z)=\begin{bmatrix} Z\\ V\end{bmatrix}$. Furthermore, $Q$ is
\[\begin{tikzcd}
U&V&Z\arrow[l]
\end{tikzcd}\] and $A\cong F Q$.
  \item If $p\geq 3$, $p\nmid n$, and $n$ is even, then $A$ is semisimple.
\end{enumerate}
\end{lemma}
\begin{proof} Let $W=\I_n$. We refer the reader to the general theory of the parametrisation of the simple $A$-modules in Section \ref{S:Descent}. The subsets $\{s_1\}$ and $\{s_2\}$ are conjugate in $W$ if and only if $n$ is odd. Furthermore, for $i\in\{1,2\}$, we have \[[\N_W(W_{\{s_i\}}):W_{\{s_i\}}]=\left \{\begin{array}{ll} 2& \text{if $n$ is even,}\\ 1&\text{otherwise.}\end{array}\right .\]

Suppose first that $n$ is even. If $2\neq p\nmid n$, we have 4 simple $A$-modules and hence $A$ is semisimple. This gives part (vi). If $2\neq p\mid n$, there are 3 simple $A$-modules $U,V,Z$ where $U,V$ are projective and $\Rad(A)=\Span\{x_3\}$. Using Lemma \ref{L:TypeITable}, the primitive orthogonal idempotents of $A$ are
\begin{align*}
  e_U&=\frac{1}{2}x_1-\frac{1}{4}x_3,\quad
  e_V=\frac{1}{2}x_2-\frac{1}{4}x_3,\quad
  e_Z=1-e_U-e_V.
\end{align*} We have $e_Zx_3e_Z=x_3$ and hence $\Proj(Z)=\begin{bmatrix} Z\\ Z\end{bmatrix}$. We must have $A\cong F Q/(\epsilon^2)$ as in part (iv). Now let $p=2$. The algebra $A$ is commutative, there is a unique simple $A$-module $U$ and $\Rad(A)=\Span\{x_1,x_2,x_3\}$. Furthermore, $x_1^2=(\ell-1)x_3=x_2^2$ and $x_1x_2=\ell x_3=x_2x_1$. So $\Rad^2(A)=\Span\{x_3\}$. Taking $\alpha=x_1+x_2$ and $\beta=x_1$, we obtain part (ii).

The proof for the case when $n$ is odd is similar. We leave it to the reader. 
\end{proof}

\begin{theorem}\label{T:RepTypeI} The algebra $\Des_F(\I_n)$ is tame if $p=2$ and $n$ is even,  and has finite representation type otherwise.
\end{theorem}
\begin{proof} Let $A=\Des_F(\I_n)$. Suppose that we are not in the case when $p=2$ and $n$ is even. By using Lemma \ref{L:TypeIsimple}, it can be checked that $\Rad^2(A)=0$ and that the  separated quiver is a disjoint union of Dynkin diagrams of type $\A$. Therefore, $A$ has finite type.

Suppose now that $p=2$ and $n$ is even,  i.e., we are in the case of Lemma \ref{L:TypeIsimple}(ii). Let $Q$ be the Ext-quiver of $A$. If $\ell$ is even, we have $A\cong F Q/(\alpha\beta-\beta\alpha,\alpha^2,\beta^2-\alpha\beta)\cong \Des_F(\B_2)$ (see Lemma \ref{L:Q22}). By Theorem \ref{T:RepTypeB}, $A$ is tame. If $\ell$ is odd, $A\cong F Q/(\alpha\beta-\beta\alpha,\alpha^2,\beta^2)$. We have the injective hull $\mathsf{I}(U)$ is isomorphic to $\Proj(U)$. Use the same method as for $\Des_F(\B_2)$ (see the proof of Lemma \ref{L:Q22}(i)), we get that $A$ is tame.
\end{proof}

\section{The Representation Type of the Descent Algebras of types $\E$, $\FF$ and $\HH$}\label{S:RepTypeEFH}

After we have classified the representation type of the descent algebras for Coxeter groups of infinite families (including type $\A$ in \cite{Erdmann/Lim:2025}), we take the opportunity to complete the classification for the finite family of types $\E$, $\FF$ and $\HH$. For this purpose, again, we make use of Magma \cite{Magma} to compute their Ext-quivers. With the similar reasons we have explained earlier (see the paragraph right before Lemma \ref{L:DExt2}), there are only finite number of Ext-quivers to compute despite there are infinitely many primes. However, unfortunately, the authors have not managed to compute for the $\E_8$ case. 

We remind the reader about the matrix $J(i)$ and the number $n_W$ defined in \cite[Section 9]{Benson/Lim:2025a}. See also Corollary \ref{C:BensonLimExtQuiver}. The values for $n_{\FF_4}$, $n_{\HH_3}$ and $n_{\HH_4}$ are easy because the descent algebras $\Des_\Q(W)$ are 2-nilpotent. For $\E_6$ and $\E_7$, the matrices $J(i)$'s are large and computing the exact greatest common divisor $d_{W,i}$ of the determinants of all its $(s\times s)$-submatrices (where $s$ is the rank of $J(i)$) is  not feasible. As such, we only compute the determinant $d$ of a particular non-singular $(s\times s)$-submatrix so that $d_{W,i}\mid d$ which gives us an upper bound for $d_{W,i}$. This is sufficient for the analyses of both $\Des_F(\E_6)$ and $\Des_F(\E_7)$. 

\begin{table}[H]
\centering
\renewcommand{\arraystretch}{1.2}
\begin{tabular}{|c|c|c|c|c|c|c|c|}
\hline
$W$ & \multicolumn{4}{c|}{what $d_{W,i}$ divides} & $|W|$ & $n_W$& primes \\ \cline{2-5}
&$i=1$&$i=2$&$i=3$&$i=4$&&& dividing $n_{W}$\\ \hline
$\FF_4$ & $1$ &-&-&-& $2^7\cdot 3^2$& $|W|$ & $2,3$ \\ \hline
$\HH_3$ & $1$ &-&-&-&  $2^3\cdot 3\cdot 5$&  $|W|$ & $2,3,5$\\ \hline 
$\HH_4$ & $1$&-&-&-& $2^6\cdot 3^2\cdot 5^2$& $|W|$ &$2,3,5$\\ \hline
$\E_6$& $1$& $2^{36}\cdot 3^{10}\cdot 5^{11}$&  $2^{33}\cdot 3^{8}\cdot 5^{6}$& $2^{9}\cdot 3^{4}\cdot 5^{2}$ & $2^7\cdot 3^4\cdot 5$&?& $2,3,5$\\ \hline 
$\E_7$& $1$& $2^{105}\cdot 3^{33}\cdot 5^7$& $2^{10}\cdot 3^{4}\cdot 5^3$& -& $2^{10}\cdot 3^{4}\cdot 5\cdot 7$&?& $2,3,5,7$\\ \hline
\end{tabular}
\caption{Primes dividing $n_{W}$ for $W=\FF_4,\HH_3,\HH_4,\E_6,\E_7$}\label{Table:n_FH}
\end{table}

Again, in Table \ref{Table:n_FH}, the primes that divides $n_W$ are precisely the primes that divides $|W|$. So we obtain their Ext-quivers as in Figures \ref{F:TypeF4p2}--\ref{F:TypeE7p0} for various $p$. We summarise our finding about the representation type in the statement below.

\begin{theorem}\label{T:RepTypeEFH} Let $X$ be either $\FF_4$, $\HH_3$, $\HH_4$, $\E_6$ or $\E_7$, and let $A=\Des_F(X)$.
\begin{enumerate}[(i)]
\item If $X=\FF_4$, then $A$ is tame if $p\geq 5$, and wild otherwise.
\item If $X=\HH_3$, then $A$ is wild if $p=2$, and tame otherwise. 
\item If $X=\HH_4$, then $A$ is wild.
\item If $X=\E_6$, then $A$ is wild.
\item If $X=\E_7$, then $A$ is wild.
\end{enumerate}  
\end{theorem}
\begin{proof} Let $Q$ be the Ext-quiver of $A$ as in Appendices \ref{Appendix:F}, \ref{Appendix:H} and \ref{Appendix:E6}. The strategy of the proof is similar with the previous sections. Firstly, the algebra $A$ is wild if its Ext-quiver meets the hypothesis of Corollary \ref{C:wildquotient}. We leave it to the reader to check that this scenario happens for the following cases:
\begin{enumerate}[(i)]
  \item $X=\FF_4$ and $p=2,3$,
  \item $X=\HH_3$ and $p=2$,
  \item $X=\HH_4$ and $p$ is arbitrary,
  \item $X=\E_6$ and $p$ is arbitrary,
  \item $X=\E_7$ and $p$ is arbitrary. 
\end{enumerate} We then need to deal with the remaining on case-by-case basis.

(i) We have $X=\FF_4$. If $p\geq 5$, $A$ is 2-nilpotent (because $Q$ has no path of length 2) and its separated quiver is a disjoint union of Dynkin diagram and $\tilde{\A}_2$. So $A$ is tame. 

(ii) We have $X=\HH_3$. Suppose that $p\geq 3$. Notice that $\dim_FA=8$. From the Ext-quiver of $A$, we observe that the total number of vertices and arrows of is also 8. Therefore, $A$ is 2-nilpotent using Theorem \ref{T:Gabriel}. The separated quiver is a disjoint union of Dynkin diagram and $\tilde{\A}_2$. So $A$ is tame.
\end{proof}

\appendix

\newpage
In the following appendices, we present the Ext-quivers for $\Des_F(W)$ for various Coxeter groups $W$ and primes $p$. For easy access of data, we have also included the Magma outputs for some Ext-quivers. For each $p$, a triplet $<a,b,c>$ denotes $c$ arrows from the vertex $a$ to $b$.

\section{The Ext Quivers of the Descent Algebra $\Des_F(\D_n)$ for $n=4,5,6$ and $p$ is Odd}\label{Appendix:ExtD}

\begin{figure}[h!]
\[\begin{tikzcd}[cells={nodes={draw=teal, thick, circle, minimum size=7mm, inner sep=0, fill=blue!20}},every arrow/.append style={-stealth},row sep=1em, column sep=1em]
    &6\ar[d]&&&10\ar[r,shift left=.5ex]\ar[r,shift right=.5ex]&2\ar[in=-15,out=15,loop]\\
    7\ar[r]&1&9\ar[l]& &3&4&5&8
\end{tikzcd}\]
\caption{The Ext-quiver of $\Des_F(\D_4)$ when $p=3$}
\label{F:TypeDn4p3}
\end{figure}

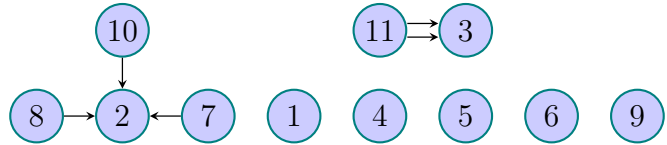
\begin{figure}[H]
\centering
\[\begin{tikzcd}[cells={nodes={draw=teal, thick, circle, minimum size=7mm, inner sep=0, fill=blue!20}},every arrow/.append style={-stealth},row sep=1em, column sep=1em]
    &10\ar[d]&&&11\ar[r,shift left=.5ex]\ar[r,shift right=.5ex]&3\\
    8\ar[r]&2&7\ar[l]& 1&4&5&6&9
\end{tikzcd}\]
\caption{The Ext-quiver of $\Des_F(\D_4)$ when $p\geq 5$}
\label{F:TypeDn4p0}
\end{figure}



\begin{figure}[ht]
\[\begin{tikzcd}[cells={nodes={draw=teal, thick, circle, minimum size=7mm, inner sep=0, fill=blue!20}},every arrow/.append style={-stealth},row sep=1em, column sep=1em]
&&11\ar[d]\ar[r]&5\\
10\ar[r]&1\ar[in=105,out=75,loop]&7\ar[l]\ar[r]\ar[d]&3\ar[dl]&&6\ar[dlll]\\
&9\ar[r]&4\ar[r]\ar[in=285,out=255,loop]&2&&&8\ar[in=105,out=75,loop]
\end{tikzcd}\]
\caption{The Ext-quiver for $\Des_F(\D_5)$ when $p=3$. The total numbers of vertices and arrows are 11 and 13 respectively. The following is the Magma output. \\
$[ <1, 1, 1>, <3, 4, 1>, <4, 2, 1>, <4, 4, 1>, <6, 4, 1>, <7, 1, 1>, <7, 3, 1>,
<7, 4, 1>, <8, 8, 1>, <9, 4, 1>, <10, 1, 1>, <11, 5, 1>, <11, 7, 1> ]$}
\label{F:TypeDn5p3}
\end{figure}
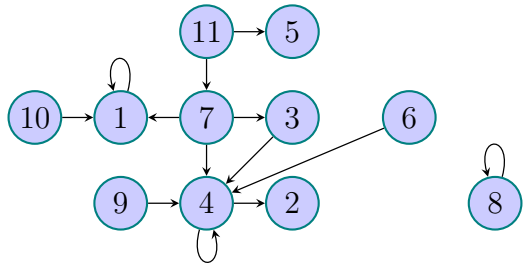

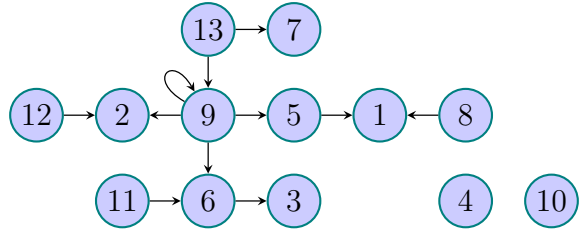
\begin{figure}[H]
\centering
\[\begin{tikzcd}[cells={nodes={draw=teal, thick, circle, minimum size=7mm, inner sep=0, fill=blue!20}},every arrow/.append style={-stealth},row sep=1em, column sep=1em]
    &&13\ar[r]\ar[d]&7\\
    12\ar[r]&2&9\ar[l]\ar[r]\ar[d]\ar[in=120,out=150,loop]&5\ar[r]&1&8\ar[l]\\
    &11\ar[r]&6\ar[r]&3&&4&10
\end{tikzcd}\]
\caption{The Ext-quiver for $\Des_F(\D_5)$ when $p=5$. The total numbers of vertices and arrows are 13 and 11 respectively. The following is the Magma output. \\
$[ <5, 1, 1>, <6, 3, 1>, <8, 1, 1>, <9, 2, 1>, <9, 5, 1>, <9, 6, 1>, <9, 9, 1>,
<11, 6, 1>, <12, 2, 1>, <13, 7, 1>, <13, 9, 1> ]$}
\label{F:TypeDn5p5}
\end{figure}

\vspace{-.8cm}
\begin{figure}[H]
\centering
\[\begin{tikzcd}[cells={nodes={draw=teal, thick, circle, minimum size=7mm, inner sep=0, fill=blue!20}},every arrow/.append style={-stealth},row sep=1em, column sep=1em]
    &&14\ar[r]\ar[d]&8\\
    13\ar[r]&3&10\ar[l]\ar[r]\ar[d]&6\ar[r]&2&9\ar[l]\\
    &12\ar[r]&7\ar[r]&4&1&5&11
\end{tikzcd}\]
\caption{The Ext-quiver for $\Des_F(\D_5)$ when $p\geq 7$. The total numbers of vertices and arrows are 14 and 10 respectively. The following is the Magma output. \\
$[ <6, 2, 1>, <7, 4, 1>, <9, 2, 1>, <10, 3, 1>, <10, 6, 1>, <10, 7, 1>, <12, 7,
1>, <13, 3, 1>, <14, 8, 1>, <14, 10, 1> ]$}
\label{F:TypeDn5p0}
\end{figure}
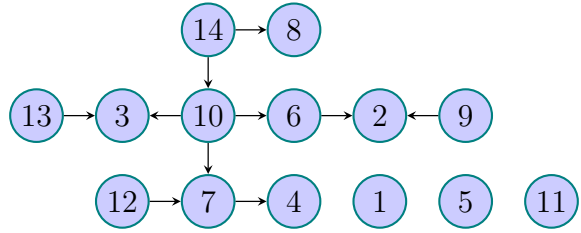

\vspace{-.8cm}
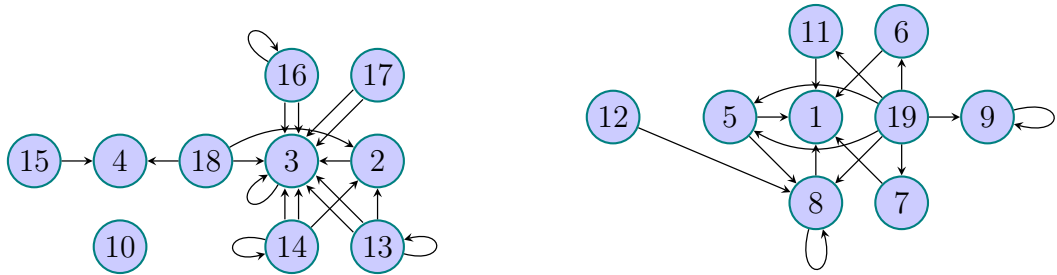
\begin{figure}[H]
\centering
\begin{minipage}{0.48\textwidth}
\[\begin{tikzcd}[cells={nodes={draw=teal, thick, circle, minimum size=7mm, inner sep=0, fill=blue!20}},every arrow/.append style={-stealth},row sep=1em, column sep=1em]
&&&16\ar[in=120,out=150,loop]\ar[d,shift left=.5ex]\ar[d,shift right=.5ex]&17\ar[dl,shift left=.5ex]\ar[dl,shift right=.5ex]\\
15\ar[r]&4&18\ar[l]\ar[r]\ar[rr,bend left]&3\ar[in=210,out=240,loop]&2\ar[l]\\
&10&&14\ar[ur]\ar[in=195,out=165,loop]\ar[u,shift left=.5ex]\ar[u,shift right=.5ex]&13\ar[u]\ar[in=15,out=-15,loop]\ar[ul,shift left=.5ex]\ar[ul,shift right=.5ex]
\end{tikzcd}\]
\end{minipage}
\hfill
\centering
\begin{minipage}{0.48\textwidth}
\[\begin{tikzcd}[cells={nodes={draw=teal, thick, circle, minimum size=7mm, inner sep=0, fill=blue!20}},every arrow/.append style={-stealth},row sep=1em, column sep=1em]
 &&&11\ar[d]&6\ar[dl]\\
 12\ar[drrr]&&5\ar[r]\ar[dr]&1&19\ar[ll,bend left]\ar[ll,bend right]\ar[r]\ar[u]\ar[d]\ar[dl]\ar[ul]&9\ar[in=-15,out=15,distance=4ex]\\
 &&&8\ar[in=285,out=255,distance=4ex]\ar[u]&7\ar[ul]
\end{tikzcd}\]
\end{minipage}
\caption{The Ext-quiver for $\Des_F(\D_6)$ when $p=3$. The total numbers of vertices and arrows are 19 and 35 respectively. The following is the Magma output. \\
$[ <2, 3, 1>, <3, 3, 1>, <5, 1, 1>, <5, 8, 1>, <6, 1, 1>, <7, 1, 1>, <8, 1, 1>,
<8, 8, 1>, <9, 9, 1>, <11, 1, 1>, <12, 8, 1>, <13, 2, 1>, <13, 3, 2>, <13, 13,
1>, <14, 2, 1>, <14, 3, 2>, <14, 14, 1>, <15, 4, 1>, <16, 3, 2>, <16, 16, 1>,
<17, 3, 2>, <18, 2, 1>, <18, 3, 1>, <18, 4, 1>, <19, 5, 2>, <19, 6, 1>, <19, 7,
1>, <19, 8, 1>, <19, 9, 1>, <19, 11, 1> ]$}
\label{F:TypeDn6p3}
\end{figure}

\begin{figure}[ht]
\centering
\[\begin{tikzcd}[cells={nodes={draw=teal, thick, circle, minimum size=7mm, inner sep=0, fill=blue!20}},every arrow/.append style={-stealth},row sep=1em, column sep=1em]
    &4&7&22\ar[d,shift left=.5ex]\ar[d,shift right=.5ex]&23\ar[dl,shift left=.5ex]\ar[dl,shift right=.5ex]\\
    21\ar[r]&9&24\ar[l]\ar[r]\ar[rr,bend left]&6&5\ar[r]&1&10\ar[l]\\
    8&16&&20\ar[u,shift left=.5ex]\ar[u,shift right=.5ex]\ar[ur]&19\ar[ul,shift left=.5ex]\ar[ul,shift right=.5ex]\ar[u]\\
    &&&17\ar[d]&12\ar[dl]\\
    18\ar[r]&2&11\ar[in=120,out=150,loop]\ar[l]\ar[r]&3&25\ar[ll,bend left]\ar[ll,bend right]\ar[u]\ar[ul]\ar[d]\ar[dl]\ar[r]&15\\
    &&&14\ar[u]&13\ar[ul]
\end{tikzcd}\]
\caption{The Ext-quiver for $\Des_F(\D_6)$ when $p=5$. The total numbers of vertices and arrows are 25 and 31 respectively. The following is the Magma output. \\
$[ <5, 1, 1>, <10, 1, 1>, <11, 2, 1>, <11, 3, 1>, <11, 11, 1>, <12, 3, 1>, <13,
3, 1>, <14, 3, 1>, <17, 3, 1>, <18, 2, 1>, <19, 5, 1>, <19, 6, 2>, <20, 5, 1>,
<20, 6, 2>, <21, 9, 1>, <22, 6, 2>, <23, 6, 2>, <24, 5, 1>, <24, 6, 1>, <24, 9,
1>, <25, 11, 2>, <25, 12, 1>, <25, 13, 1>, <25, 14, 1>, <25, 15, 1>, <25, 17, 1>
]$}
\label{F:TypeDn6p5}
\end{figure}
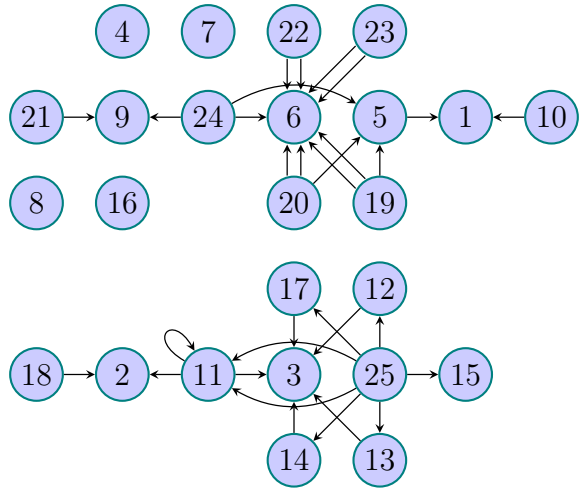


\begin{figure}[H]
\centering
\[\begin{tikzcd}[cells={nodes={draw=teal, thick, circle, minimum size=7mm, inner sep=0, fill=blue!20}},every arrow/.append style={-stealth},row sep=1em, column sep=1em]
    1&5&8&23\ar[d,shift left=.5ex]\ar[d,shift right=.5ex]&24\ar[dl,shift left=.5ex]\ar[dl,shift right=.5ex]\\
    22\ar[r]&10&25\ar[l]\ar[r]\ar[rr,bend left]&7&6\ar[r]&2&11\ar[l]\\
    9&17&&21\ar[u,shift left=.5ex]\ar[u,shift right=.5ex]\ar[ur]&20\ar[ul,shift left=.5ex]\ar[ul,shift right=.5ex]\ar[u]\\
    &&&18\ar[d]&13\ar[dl]\\
    19\ar[r]&3&12\ar[l]\ar[r]&4&26\ar[ll,bend left]\ar[ll,bend right]\ar[u]\ar[ul]\ar[d]\ar[dl]\ar[r]&16\\
    &&&15\ar[u]&14\ar[ul]
\end{tikzcd}\]
\caption{The Ext-quiver for $\Des_F(\D_6)$ when $p\geq 7$. The total numbers of vertices and arrows are 26 and 30 respectively. The following is the Magma output. \\
$[ <6, 2, 1>, <11, 2, 1>, <12, 3, 1>, <12, 4, 1>, <13, 4, 1>, <14, 4, 1>, <15, 4,
1>, <18, 4, 1>, <19, 3, 1>, <20, 6, 1>, <20, 7, 2>, <21, 6, 1>, <21, 7, 2>, <22,
10, 1>, <23, 7, 2>, <24, 7, 2>, <25, 6, 1>, <25, 7, 1>, <25, 10, 1>, <26, 12,
2>, <26, 13, 1>, <26, 14, 1>, <26, 15, 1>, <26, 16, 1>, <26, 18, 1> ]$}
\label{F:TypeDn6p0}
\end{figure}
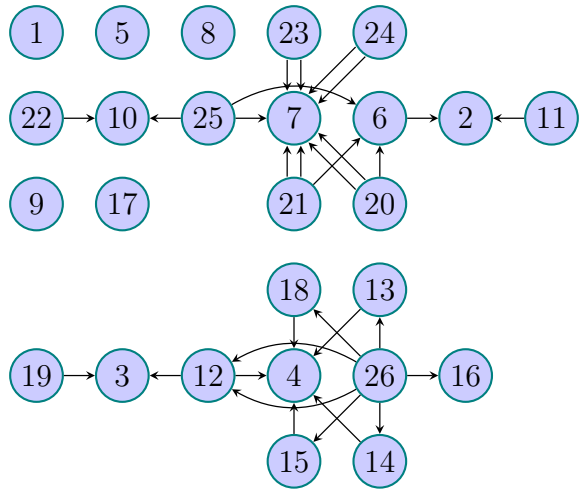

\section{The Ext Quivers of the Descent Algebra $\Des_F(\FF_4)$}\label{Appendix:F}

\begin{figure}[H]
\centering
\[\begin{tikzcd}[cells={nodes={draw=teal, thick, circle, minimum size=7mm, inner sep=0, fill=blue!20}},every arrow/.append style={-stealth},row sep=1em, column sep=1em]
1 \ar[in=15,out=-15,loop] \ar[in=105,out=75,loop] \ar[in=195,out=165,loop] \ar[in=285,out=255,loop]
\end{tikzcd}\]
\caption{The Ext-quiver for $\Des_F(\FF_4)$ when $p=2$}
\label{F:TypeF4p2}
\end{figure}

\begin{figure}[H]
\centering
\[\begin{tikzcd}[cells={nodes={draw=teal, thick, circle, minimum size=7mm, inner sep=0, fill=blue!20}},every arrow/.append style={-stealth},row sep=1em, column sep=1em]
6\ar[r]&4 \ar[in=75,out=105,loop]&
3\ar[r]&5 \ar[in=75,out=105,loop]&
7\ar[in=120,out=150,loop] \ar[in=210,out=240,loop] \ar[r,shift left=.5ex]\ar[r,shift right=.5ex] &1&2
\end{tikzcd}\]
\caption{The Ext-quiver for $\Des_F(\FF_4)$ when $p=3$}
\label{F:TypeF4p3}
\end{figure}

\begin{figure}[H]
\centering
\[\begin{tikzcd}[cells={nodes={draw=teal, thick, circle, minimum size=7mm, inner sep=0, fill=blue!20}},every arrow/.append style={-stealth},row sep=1em, column sep=1em]
8\ar[r]&2&
11\ar[r]&3&
12 \ar[r,shift left=.5ex]\ar[r,shift right=.5ex] &5\\
1&4&6&7&9&10
\end{tikzcd}\]
\caption{The Ext-quiver for $\Des_F(\FF_4)$ when $p\geq 5$}
\label{F:TypeF4p0}
\end{figure}

\section{The Ext Quivers of the Descent Algebra of type $\HH$}\label{Appendix:H}

\begin{figure}[H]
\centering
\[\begin{tikzcd}[cells={nodes={draw=teal, thick, circle, minimum size=7mm, inner sep=0, fill=blue!20}},every arrow/.append style={-stealth},row sep=1em, column sep=1em]
1 \ar[in=15,out=-15,loop] \ar[in=105,out=75,loop] \ar[in=195,out=165,loop] \ar[in=285,out=255,loop]
\end{tikzcd}\]
\caption{The Ext-quiver for $\Des_F(\HH_3)$ when $p=2$}
\label{F:TypeH3p2}
\end{figure}

\begin{figure}[H]
\centering
\[\begin{tikzcd}[cells={nodes={draw=teal, thick, circle, minimum size=7mm, inner sep=0, fill=blue!20}},every arrow/.append style={-stealth},row sep=1em, column sep=1em]
5\ar[r,shift left=.5ex]\ar[r,shift right=.5ex]&1 &4\ar[in=15,out=-15,loop]& 2&3
\end{tikzcd}\]
\caption{The Ext-quiver for $\Des_F(\HH_3)$ when $p=3$}
\label{F:TypeH3p3}
\end{figure}

\begin{figure}[H]
\centering
\[\begin{tikzcd}[cells={nodes={draw=teal, thick, circle, minimum size=7mm, inner sep=0, fill=blue!20}},every arrow/.append style={-stealth},row sep=1em, column sep=1em]
5\ar[r,shift left=.5ex]\ar[r,shift right=.5ex]&1 &2\ar[in=15,out=-15,loop]& 3&4
\end{tikzcd}\]
\caption{The Ext-quiver for $\Des_F(\HH_3)$ when $p=5$}
\label{F:TypeH3p5}
\end{figure}
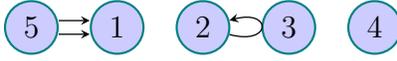

\begin{figure}[H]
\centering
\[\begin{tikzcd}[cells={nodes={draw=teal, thick, circle, minimum size=7mm, inner sep=0, fill=blue!20}},every arrow/.append style={-stealth},row sep=1em, column sep=1em]
6\ar[r,shift left=.5ex]\ar[r,shift right=.5ex]&2 & 1&3&4&5
\end{tikzcd}\]
\caption{The Ext-quiver for $\Des_F(\HH_3)$ when $p\geq 7$}
\label{F:TypeH3p0}
\end{figure}


\begin{figure}[H]

\centering
\[\begin{tikzcd}[cells={nodes={draw=teal, thick, circle, minimum size=7mm, inner sep=0, fill=blue!20}},every arrow/.append style={-stealth},row sep=1em, column sep=1em]
1 \ar[in=15,out=-15,loop] \ar[in=105,out=75,loop] \ar[in=195,out=165,loop] \ar[in=285,out=255,loop]
\end{tikzcd}\]
\caption{The Ext-quiver for $\Des_F(\HH_4)$ when $p= 2$}
\label{F:TypeH4p2}
\end{figure}

\begin{figure}[H]
\centering
\[\begin{tikzcd}[cells={nodes={draw=teal, thick, circle, minimum size=7mm, inner sep=0, fill=blue!20}},every arrow/.append style={-stealth},row sep=1em, column sep=1em]
3\ar[r,shift left=.5ex]\ar[r,shift right=.5ex]&5 \ar[in=105,out=75,loop]&6\ar[l]&
7\ar[in=120,out=150,loop] \ar[in=210,out=240,loop] \ar[r,shift left=.5ex]\ar[r,shift right=.5ex] &2&1&4
\end{tikzcd}\]
\caption{The Ext-quiver for $\Des_F(\HH_4)$ when $p= 3$}
\label{F:TypeH4p3}
\end{figure}

\begin{figure}[H]
\centering
\[\begin{tikzcd}[cells={nodes={draw=teal, thick, circle, minimum size=7mm, inner sep=0, fill=blue!20}},every arrow/.append style={-stealth},row sep=1em, column sep=1em]
3\ar[r,shift left=.5ex]\ar[r,shift right=.5ex]&4 \ar[in=105,out=75,loop]&6\ar[l]&
2&7\ar[l]\ar[in=75,out=105,loop] \ar[r,shift left=.5ex]\ar[r,shift right=.5ex] &1&5
\end{tikzcd}\]
\caption{The Ext-quiver for $\Des_F(\HH_4)$ when $p= 5$}
\label{F:TypeH4p5}
\end{figure}

\begin{figure}[H]
\centering
\[\begin{tikzcd}[cells={nodes={draw=teal, thick, circle, minimum size=7mm, inner sep=0, fill=blue!20}},every arrow/.append style={-stealth},row sep=1em, column sep=1em]
6\ar[r,shift left=.5ex]\ar[r,shift right=.5ex]&2 &9\ar[l]&
5&10\ar[l] \ar[r,shift left=.5ex]\ar[r,shift right=.5ex] &4\\
&1&3&7&8
\end{tikzcd}\]
\caption{The Ext-quiver for $\Des_F(\HH_4)$ when $p\geq 7$}
\label{F:TypeH4p0}
\end{figure}

\section{The Ext-quiver for $\Des_F(\E_6)$}\label{Appendix:E6}

\vspace{-.8cm}
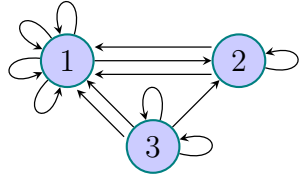
\begin{figure}[H]
\centering
\[\begin{tikzcd}[cells={nodes={draw=teal, thick, circle, minimum size=7mm, inner sep=0, fill=blue!20}},every arrow/.append style={-stealth},row sep=1em, column sep=1em]
1 \ar[in=155,out=125,loop] \ar[in=105,out=75,loop] \ar[in=205,out=175,loop] \ar[in=255,out=225,loop]\ar[rr]&&2\ar[ll,shift left=1ex]\ar[ll,shift right=1ex]\ar[in=15,out=-15,loop]\\
&3\ar[lu,shift left=1ex]\ar[lu,shift right=0ex]\ar[ru]\ar[in=15,out=-15,loop] \ar[in=105,out=75,loop]
\end{tikzcd}\]
\caption{The Ext-quiver for $\Des_F(\E_6)$ when $p=2$. The total numbers of vertices and arrows are 3 and 13 respectively. The following is the Magma output. \\
$[<1, 1, 4>, <1, 2, 1>, <2, 1, 2>, <2, 2, 1>, <3, 1, 2>, <3, 2, 1>, <3, 3, 2>]$}
\label{F:TypeE6p2}
\end{figure}

\vspace{-.8cm}
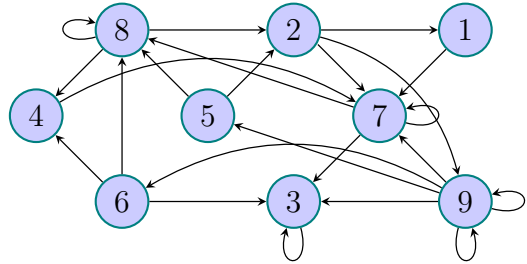
\begin{figure}[H]
\centering
\[\begin{tikzcd}[cells={nodes={draw=teal, thick, circle, minimum size=7mm, inner sep=0, fill=blue!20}},every arrow/.append style={-stealth},row sep=1em, column sep=1em]
&8\ar[in=165,out=195,loop]\ar[rr]\ar[ld]&&2\ar[rr]\ar[dr,shift left=.5ex]\ar[rrdd,bend left]&&1\ar[dl]\\
4\ar[rrrr,bend left]&&5\ar[ur]\ar[ul]&&7\ar[ld]\ar[in=15,out=-15,loop]\ar[lllu]\\
&6\ar[rr]\ar[ul]\ar[uu]&&3\ar[in=255,out=285,loop]&&9\ar[in=285,out=255,loop]\ar[in=15,out=-15,loop]\ar[lu]\ar[lllu]\ar[llll,bend right]\ar[ll]
\end{tikzcd}\]
\caption{The Ext-quiver for $\Des_F(\E_6)$ when $p=3$. The total numbers of vertices and arrows are 9 and 23 respectively. The following is the Magma output. \\ $[<1, 7, 1>, <2, 1, 1>, <2, 7, 1>, <2, 9, 1>, <3, 3, 1>, <4, 7, 1>, <5, 2, 1>,
<5, 8, 1>, <6, 3, 1>, <6, 4, 1>, <6, 8, 1>, <7, 3, 1>, <7, 7, 1>, <7, 8, 1>, <8,
2, 1>, <8, 4, 1>, <8, 8, 1>, <9, 3, 1>, <9, 5, 1>, <9, 6, 1>, <9, 7, 1>, <9, 9,
2>]$}
\label{F:TypeE6p3}
\end{figure}

\vspace{-.8cm}
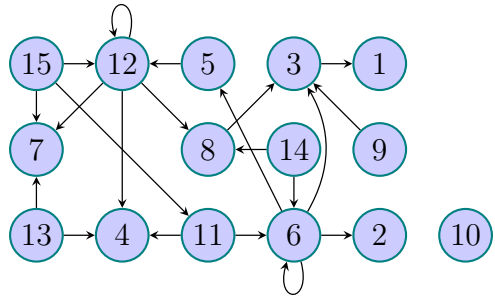
\begin{figure}[H]
\centering
\[\begin{tikzcd}[cells={nodes={draw=teal, thick, circle, minimum size=7mm, inner sep=0, fill=blue!20}},every arrow/.append style={-stealth},row sep=1em, column sep=1em]
15\ar[d]\ar[r]\ar[rrdd]&12\ar[in=105,out=75,loop]\ar[dl]\ar[dd]\ar[dr]&5\ar[l]&3\ar[r]&1\\
7&&8\ar[ru]&14\ar[l]\ar[d]&9\ar[ul]\\
13\ar[u]\ar[r]&4&11\ar[l]\ar[r]&6\ar[r]\ar[luu]\ar[uu,bend right]\ar[in=255,out=285,loop]&2&10
\end{tikzcd}\]
\caption{The Ext-quiver for $\Des_F(\E_6)$ when $p=5$. The total numbers of vertices and arrows are 15 and 21 respectively. The following is the Magma output. \\
$[<3, 1, 1>, <5, 12, 1>, <6, 2, 1>, <6, 3, 1>, <6, 5, 1>, <6, 6, 1>, <8, 3, 1>,
<9, 3, 1>, <11, 4, 1>, <11, 6, 1>, <12, 4, 1>, <12, 7, 1>, <12, 8, 1>, <12, 12,
1>, <13, 4, 1>, <13, 7, 1>, <14, 6, 1>, <14, 8, 1>, <15, 7, 1>, <15, 11, 1>,
<15, 12, 1>]$}
\label{F:TypeE6p5}
\end{figure}

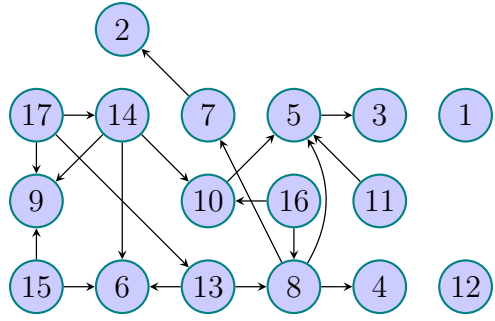
\begin{figure}[H]
\centering
\[\begin{tikzcd}[cells={nodes={draw=teal, thick, circle, minimum size=7mm, inner sep=0, fill=blue!20}},every arrow/.append style={-stealth},row sep=1em, column sep=1em]
&2\\
17\ar[d]\ar[r]\ar[rrdd]&14\ar[dl]\ar[dd]\ar[dr]&7\ar[lu]&5\ar[r]&3&1\\
9&&10\ar[ru]&16\ar[l]\ar[d]&11\ar[ul]&\\
15\ar[u]\ar[r]&6&13\ar[l]\ar[r]&8\ar[r]\ar[luu]\ar[uu,bend right]&4&12
\end{tikzcd}\]
\caption{The Ext-quiver for $\Des_F(\E_6)$ when $p\geq 7$. The total numbers of vertices and arrows are 17 and 19 respectively. The following is the Magma output. \\ $[<5, 3, 1>, <7, 2, 1>, <8, 4, 1>, <8, 5, 1>, <8, 7, 1>, <10, 5, 1>, <11, 5, 1>,
<13, 6, 1>, <13, 8, 1>, <14, 6, 1>, <14, 9, 1>, <14, 10, 1>, <15, 6, 1>, <15, 9,
1>, <16, 8, 1>, <16, 10, 1>, <17, 9, 1>, <17, 13, 1>, <17, 14, 1>]$}
\label{F:TypeE6p0}
\end{figure}

\section{The Ext-quiver for $\Des_F(\E_7)$}\label{Appendix:E7}

\vspace{-.5cm}
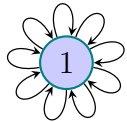
\begin{figure}[H]
\centering
\[\begin{tikzcd}[cells={nodes={draw=teal, thick, circle, minimum size=7mm, inner sep=0, fill=blue!20}},every arrow/.append style={-stealth},row sep=1em, column sep=1em]
1 \ar[in=0,out=30,loop] \ar[in=40,out=70,loop] \ar[in=80,out=110,loop] \ar[in=120,out=150,loop] \ar[in=160,out=190,loop] \ar[in=200,out=230,loop] \ar[in=240,out=270,loop] \ar[in=280,out=310,loop] \ar[in=320,out=350,loop]
\end{tikzcd}\]
\caption{The Ext-quiver for $\Des_F(\E_7)$ when $p=2$. The total numbers of vertices and arrows are 1 and 9 respectively. The following is the Magma output. \\ $[<1, 1, 9>]$}
\label{F:TypeE7p2}
\end{figure}

\vspace{-1cm}
\begin{figure}[H]
\centering
\[\begin{tikzcd}[cells={nodes={draw=teal, thick, circle, minimum size=7mm, inner sep=0, fill=blue!20}},every arrow/.append style={-stealth}]
&&13\ar[in=120,out=90,loop]\ar[dl,shift left=-.2ex]\ar[dl,shift right=1.2ex]\\
14\ar[r,shift left=.5ex]\ar[r,shift right=.5ex]\ar[dr]\ar[in=75,out=105,loop] &1\ar[in=105,out=75,loop]&15\ar[l]\ar[dl,shift left=.7ex]\ar[dl,shift right=.5ex]\ar[dr,shift left=.5ex] \ar[d]&16\ar[ll,bend right]\ar[dll,shift left=.5ex]\ar[dll,shift right=.5ex]\ar[d]\ar[dl,shift left=.5ex]\ar[dl,shift right=.5ex]\ar[ul]\\
11\ar[ur]\ar[r]\ar[u]&2\ar[u]&12\ar[ul,shift left=.5ex]\ar[ul,shift right=.5ex] \ar[l,shift left=.3ex]\ar[r,shift right=.3ex]\ar[in=285,out=255,loop]&10\ar[ull,shift left=.7ex]\ar[ull,shift right=.3ex]\ar[ll,bend left]\ar[in=285,out=255,loop]\ar[in=15,out=-15,loop]\ar[l,shift left=1ex]\\
6\ar[rdd]&
4\ar[ld,shift right=1ex]\ar[dd,bend left]
&3\ar[dll,shift right=1ex]\ar[ddl,shift left=.5ex]\ar[d]\\
5\ar[in=165,out=195,loop]\ar[dr]&
17\ar[ur]\ar[u,shift left=.5ex]\ar[u,shift right=.5ex]\ar[u,shift left=1.5ex]\ar[u,shift right=1.5ex] \ar[l,shift left=.5ex]\ar[l,shift right=.5ex] \ar[ul,shift left=.5ex]\ar[ul,shift right=.5ex] \ar[d]\ar[r,shift left=.5ex]\ar[r,shift right=.5ex] \ar[in=310,out=280,loop] \ar[in=230,out=260,loop]
&8 \ar[ll,bend right]\ar[ll,bend left] \ar[in=-15,out=15,loop]\\
9\ar[u]&7\ar[in=-15,out=15,loop] \ar[ul,shift left=1ex]\ar[ul,shift left=2ex] \ar[ur,shift right=.5ex]
\end{tikzcd}\]
\caption{The Ext-quiver for $\Des_F(\E_7)$ when $p=3$. The total numbers of vertices and arrows are 17 and 65 respectively. The following is the Magma output.\\
$[<1, 1, 1>, <2, 1, 1>, <3, 5, 1>, <3, 7, 1>, <3, 8, 1>, <4, 5, 1>, <4, 7, 1>,
<5, 5, 1>, <5, 7, 1>, <6, 7, 1>, <7, 5, 2>, <7, 7, 1>, <7, 8, 1>, <8, 5, 2>, <8,
8, 1>, <9, 5, 1>, <10, 1, 2>, <10, 2, 1>, <10, 10, 2>, <10, 12, 1>, <11, 1, 1>,
<11, 2, 1>, <11, 14, 1>, <12, 1, 2>, <12, 2, 1>, <12, 10, 1>, <12, 12, 1>, <13,
1, 2>, <13, 13, 1>, <14, 1, 2>, <14, 2, 1>, <14, 14, 1>, <15, 1, 1>, <15, 2, 2>,
<15, 10, 1>, <15, 12, 1>, <16, 1, 1>, <16, 2, 2>, <16, 10, 1>, <16, 12, 2>, <16,
13, 1>, <17, 3, 1>, <17, 4, 4>, <17, 5, 2>, <17, 6, 2>, <17, 7, 1>, <17, 8, 2>,
<17, 17, 2>]$}
\label{F:TypeE7p3}
\end{figure}
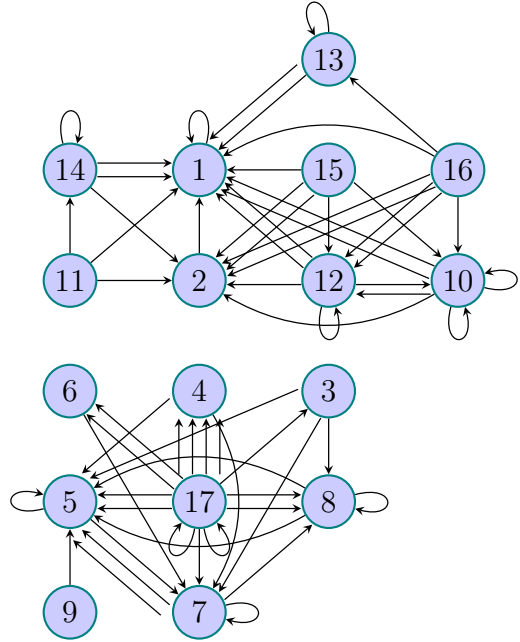

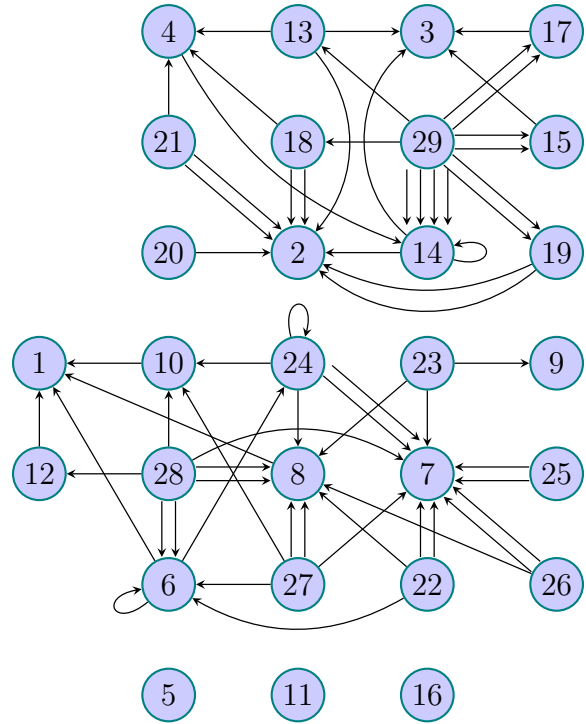
\begin{figure}[H]
\centering
\[\begin{tikzcd}[cells={nodes={draw=teal, thick, circle, minimum size=7mm, inner sep=0, fill=blue!20}},every arrow/.append style={-stealth}]
&4\ar[rrdd,bend right=20]&13\ar[l]\ar[r]\ar[dd,bend left=40]&3&17\ar[l]\\
&21\ar[u]\ar[dr,shift left=.5ex]\ar[dr,shift right=.5ex]&18\ar[ul]\ar[d,shift left=.5ex]\ar[d,shift right=.5ex]&29\ar[r,shift left=.5ex]\ar[r,shift right=.5ex]\ar[lu,shift right=.3ex]\ar[l]\ar[ur,shift left=.5ex]\ar[ur,shift right=.5ex]\ar[dr,shift left=.5ex]\ar[dr,shift right=.5ex]\ar[d,shift left=.5ex]\ar[d,shift right=.5ex]\ar[d,shift left=1.5ex]\ar[d,shift right=1.5ex]&15\ar[ul]\\
&20\ar[r]&2&14\ar[in=15,out=-15,loop]\ar[uu,bend left=50]\ar[l]&19\ar[ll,bend left=25]\ar[ll,bend left=40]\\
1&10\ar[l]&24\ar[in=75,out=105,loop]\ar[l]\ar[d]\ar[dr,shift left=1.5ex]\ar[dr,shift right=-.5ex]&23\ar[dl]\ar[d]\ar[r]&9\\
12\ar[u]&28\ar[u]\ar[d,shift left=.5ex]\ar[d,shift right=.5ex]\ar[r,shift left=.5ex]\ar[r,shift right=.5ex]\ar[l]\ar[rr,bend left]&8\ar[llu]&7&25\ar[l,shift left=.5ex]\ar[l,shift right=.5ex]\\
&6\ar[in=190,out=220,loop]\ar[uul]\ar[uur]&27\ar[l]\ar[uul]\ar[ur]\ar[u,shift left=.5ex]\ar[u,shift right=.5ex]&22\ar[lu]\ar[ll,bend left]\ar[u,shift left=.5ex]\ar[u,shift right=.5ex]&26\ar[ul,shift left=.5ex]\ar[ul,shift right=.5ex]\ar[ull]\\
&5&11&16
\end{tikzcd}\]
\caption{The Ext-quiver for $\Des_F(\E_7)$ when $p=5$. The total numbers of vertices and arrows are 29 and 65 respectively. The following is the Magma output.\\
$[<4, 14, 1>, <6, 1, 1>, <6, 6, 1>, <6, 24, 1>, <8, 1, 1>, <10, 1, 1>, <12, 1,
1>, <13, 2, 1>, <13, 3, 1>, <13, 4, 1>, <14, 2, 1>, <14, 3, 1>, <14, 14, 1>,
<15, 3, 1>, <17, 3, 1>, <18, 2, 2>, <18, 4, 1>, <19, 2, 2>, <20, 2, 1>, <21, 2,
2>, <21, 4, 1>, <22, 6, 1>, <22, 7, 2>, <22, 8, 1>, <23, 7, 1>, <23, 8, 1>, <23,
9, 1>, <24, 7, 2>, <24, 8, 1>, <24, 10, 1>, <24, 24, 1>, <25, 7, 2>, <26, 7, 2>,
<26, 8, 1>, <27, 6, 1>, <27, 7, 1>, <27, 8, 2>, <27, 10, 1>, <28, 6, 2>, <28, 7,
1>, <28, 8, 2>, <28, 10, 1>, <28, 12, 1>, <29, 13, 1>, <29, 14, 4>, <29, 15, 2>,
<29, 17, 2>, <29, 18, 1>, <29, 19, 2>]$}
\label{F:TypeE7p5}
\end{figure}

\begin{figure}[H]
\centering
\[\begin{tikzcd}[cells={nodes={draw=teal, thick, circle, minimum size=7mm, inner sep=0, fill=blue!20}},every arrow/.append style={-stealth}]
1&6\ar[l]&15\ar[l]\ar[r]\ar[dd,bend left=40]&5&19\ar[l]\\
&23\ar[u]\ar[dr,shift left=.5ex]\ar[dr,shift right=.5ex]&20\ar[ul]\ar[d,shift left=.5ex]\ar[d,shift right=.5ex]&31\ar[r,shift left=.5ex]\ar[r,shift right=.5ex]\ar[lu,shift right=.3ex]\ar[l]\ar[ur,shift left=.5ex]\ar[ur,shift right=.5ex]\ar[dr,shift left=.5ex]\ar[dr,shift right=.5ex]\ar[d,shift left=.5ex]\ar[d,shift right=.5ex]\ar[d,shift left=1.5ex]\ar[d,shift right=1.5ex]&17\ar[ul]\\
&22\ar[r]&4&16\ar[uu,bend left=50]\ar[l]&21\ar[ll,bend left=25]\ar[ll,bend left=40]\\
2&12\ar[l]&26\ar[l]\ar[d]\ar[dr,shift left=1.5ex]\ar[dr,shift right=-.5ex]&25\ar[dl]\ar[d]\ar[r]&11\\
14\ar[u]&30\ar[u]\ar[d,shift left=.5ex]\ar[d,shift right=.5ex]\ar[r,shift left=.5ex]\ar[r,shift right=.5ex]\ar[l]\ar[rr,bend left]&10\ar[llu]&9&27\ar[l,shift left=.5ex]\ar[l,shift right=.5ex]\\
3&8\ar[uul]\ar[l]&29\ar[in=-15,out=15,loop]\ar[l]\ar[uul]\ar[ur]\ar[u,shift left=.5ex]\ar[u,shift right=.5ex]&24\ar[lu]\ar[ll,bend left]\ar[u,shift left=.5ex]\ar[u,shift right=.5ex]&28\ar[ul,shift left=.5ex]\ar[ul,shift right=.5ex]\ar[ull]\\
&7&13&18
\end{tikzcd}\]
\caption{The Ext-quiver for $\Des_F(\E_7)$ when $p=7$. The total numbers of vertices and arrows are 31 and 63 respectively. The following is the Magma output.\\
$[<6, 1, 1>, <8, 2, 1>, <8, 3, 1>, <10, 2, 1>, <12, 2, 1>, <14, 2, 1>, <15, 4,
1>, <15, 5, 1>, <15, 6, 1>, <16, 4, 1>, <16, 5, 1>, <17, 5, 1>, <19, 5, 1>, <20,
4, 2>, <20, 6, 1>, <21, 4, 2>, <22, 4, 1>, <23, 4, 2>, <23, 6, 1>, <24, 8, 1>,
<24, 9, 2>, <24, 10, 1>, <25, 9, 1>, <25, 10, 1>, <25, 11, 1>, <26, 9, 2>, <26,
10, 1>, <26, 12, 1>, <27, 9, 2>, <28, 9, 2>, <28, 10, 1>, <29, 8, 1>, <29, 9,
1>, <29, 10, 2>, <29, 12, 1>, <29, 29, 1>, <30, 8, 2>, <30, 9, 1>, <30, 10, 2>,
<30, 12, 1>, <30, 14, 1>, <31, 15, 1>, <31, 16, 4>, <31, 17, 2>, <31, 19, 2>,
<31, 20, 1>, <31, 21, 2>]$}
\label{F:TypeE7p7}
\end{figure}
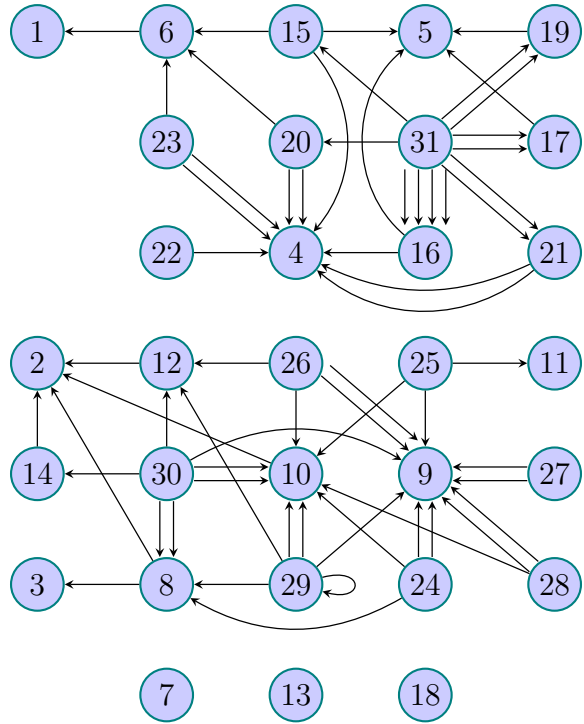

\begin{figure}[H]
\centering
\[\begin{tikzcd}[cells={nodes={draw=teal, thick, circle, minimum size=7mm, inner sep=0, fill=blue!20}},every arrow/.append style={-stealth}]
2&7\ar[l]&16\ar[l]\ar[r]\ar[dd,bend left=40]&6&20\ar[l]\\
&24\ar[u]\ar[dr,shift left=.5ex]\ar[dr,shift right=.5ex]&21\ar[ul]\ar[d,shift left=.5ex]\ar[d,shift right=.5ex]&32\ar[r,shift left=.5ex]\ar[r,shift right=.5ex]\ar[lu, shift right=.3ex]\ar[l]\ar[ur,shift left=.5ex]\ar[ur,shift right=.5ex]\ar[dr,shift left=.5ex]\ar[dr,shift right=.5ex]\ar[d,shift left=.5ex]\ar[d,shift right=.5ex]\ar[d,shift left=1.5ex]\ar[d,shift right=1.5ex]&18\ar[ul]\\
&23\ar[r]&5&17\ar[uu,bend left=50]\ar[l]&22\ar[ll,bend left=25]\ar[ll,bend left=40]\\
3&13\ar[l]&27\ar[l]\ar[d]\ar[dr,shift left=1.5ex]\ar[dr,shift right=-.5ex]&26\ar[dl]\ar[d]\ar[r]&12\\
15\ar[u]&31\ar[u]\ar[d,shift left=.5ex]\ar[d,shift right=.5ex]\ar[r,shift left=.5ex]\ar[r,shift right=.5ex]\ar[l]\ar[rr,bend left]&11\ar[llu]&10&28\ar[l,shift left=.5ex]\ar[l,shift right=.5ex]\\
4&9\ar[uul]\ar[l]&30\ar[l]\ar[uul]\ar[ur]\ar[u,shift left=.5ex]\ar[u,shift right=.5ex]&25\ar[lu]\ar[ll,bend left]\ar[u,shift left=.5ex]\ar[u,shift right=.5ex]&29\ar[ul,shift left=.5ex]\ar[ul,shift right=.5ex]\ar[ull]\\
1&8&14&19
\end{tikzcd}\]
\caption{The Ext-quiver for $\Des_F(\E_7)$ when $p\geq 11$. The total numbers of vertices and arrows are 32 and 62 respectively. The following is the Magma output.\\
$[<7, 2, 1>, <9, 3, 1>, <9, 4, 1>, <11, 3, 1>, <13, 3, 1>, <15, 3, 1>, <16, 5,
1>, <16, 6, 1>, <16, 7, 1>, <17, 5, 1>, <17, 6, 1>, <18, 6, 1>, <20, 6, 1>, <21,
5, 2>, <21, 7, 1>, <22, 5, 2>, <23, 5, 1>, <24, 5, 2>, <24, 7, 1>, <25, 9, 1>,
<25, 10, 2>, <25, 11, 1>, <26, 10, 1>, <26, 11, 1>, <26, 12, 1>, <27, 10, 2>,
<27, 11, 1>, <27, 13, 1>, <28, 10, 2>, <29, 10, 2>, <29, 11, 1>, <30, 9, 1>,
<30, 10, 1>, <30, 11, 2>, <30, 13, 1>, <31, 9, 2>, <31, 10, 1>, <31, 11, 2>,
<31, 13, 1>, <31, 15, 1>, <32, 16, 1>, <32, 17, 4>, <32, 18, 2>, <32, 20, 2>,
<32, 21, 1>, <32, 22, 2>]$}
\label{F:TypeE7p0}
\end{figure}


\begin{thebibliography}{111}
\bibitem{Aguiar/Bergeron/Nyman:2004} M. Aguiar, N. Bergeron and K. Nyman, The peak algebra and the descent algebras of types $B$ and $D$, Trans. Amer. Math. Soc. 356 (2004), no.7, 2781--2824.
\bibitem{Atkinson/Pfeiffer/vanWilligenburg:2002a} M. D. Atkinson, G. Pfeiffer and S. J. van Willigenburg, The $p$-modular descent algebras, Algebr. Represent. Theory 5 (2002), no. 1, 101--113.
\bibitem{Atkinson/vanWilligenburg:1997a} M. D. Atkinson and S. J. van Willigenburg, The $p$-modular descent algebra of the symmetric group, Bull. London Math. Soc. 29 (1997), no. 4, 407--414.
\bibitem{Auslander/Reiten/Smalo:1995a} M. Auslander, I. Reiten and S. O. Smalø, Representation Theory of Artin Algebras, Cambridge Stud. Adv. Math., 36, Cambridge University Press, Cambridge, 1995.
\bibitem{Benson:1991} D. J. Benson, Representations and Cohomology I: Basic Representation Theory of Finite Groups and Associative Algebras, Cambridge Studies in Advanced Mathematics, vol. 30, Cambridge University Press, 1991, reprinted in paperback, 1998.
\bibitem{Benson/Lim:2025a} D. J. Benson and K. J. Lim, Projective modules and cohomology for integral basic algebras, Int. Math. Res. Not. IMRN 2025, no. 9, Paper No. rnaf105, 23 pp.
\bibitem{Bergeron/Bergeron:1992} F. Bergeron and N. Bergeron, A decomposition of the descent algebra of the hyperoctahedral group, I, J. Algebra 148 (1992), no. 1, 86--97.
\bibitem{Bergeron/Bergeron/Howlett/Taylor:1992} F. Bergeron, N. Bergeron, R. B. Howlett and D. E. Taylor, A decomposition of the descent algebra of a finite Coxeter group, J. Algebraic Combin. 1 (1992), no. 1, 23--44.
\bibitem{Bergeron/Garsia/Reutenauer:1992a}  F. Bergeron, A. Garsia, and C. Reutenauer, Homomorphisms between Solomon's descent algebras, J. Algebra 150 (1992), no. 2, 503--519.
\bibitem{Bergeron/vanWilligenburg:1998} N. Bergeron and S. J. van Willigenburg, A multiplication rule for the descent algebra of type $D$, J. Algebra 206 (1998), no.2, 699--705.
\bibitem{Blessenohl/Laue:1996a} D. Blessenohl and H. Laue, On the descending Loewy series of Solomon's descent algebra,  J. Algebra 180 (1996), no. 3, 698--724.
\bibitem{Blessenohl/Laue:2002a} D. Blessenohl and H. Laue, The module structure of Solomon's descent algebra. J. Aust. Math. Soc. 72 (2002), no. 3, 317--333.
\bibitem{Magma} W. Bosma, J. Cannon and C. Playoust, The Magma algebra system. I. The user language, J. Symbolic Comput., 24 (1997), 235--265.

\bibitem{Cline/Parshall/Scott:1996a} E. Cline, B. Parshall and L. Scott, Stratifying Endomorphism Algebras, Mem. Amer. Math. Soc. 124 (1996), no. 591.
\bibitem{CB} W. Crawley-Boevey, On tame algebras and bocses. Proc. London Math. Soc.(3) 56 (1988), no.3, 451--483.
\bibitem{Dlab/Ringel:1976a} V. Dlab and C. M. Ringel, Indecomposable representations of graphs and algebras, Mem. Amer. Math. Soc. 6 (1976), no. 173.
\bibitem{Donovan/Freislich:1973} P. Donovan and M. R.  Freislich, The Representation Theory of Finite Graphs and Associated Algebras, Carleton Mathematical Lecture Notes, No. 5. Carleton University, Ottawa, ON, 1973. iii+83 pp.
\bibitem{Drozd:1980} Ju. A. Drozd, Tame and wild matrix problems, Representation Theory, II (Proc. Second Internat. Conf., Carleton Univ., Ottawa, Ont., 1979), pp. 242–258, Lecture Notes in Math., 832, Springer, Berlin, 1980.
\bibitem{Erdmann:1990a} K. Erdmann, Blocks of Tame Representation Type and Related Algebras, Lecture Notes in Mathematics, vol. 1428, Springer-Verlag, Berlin, 1990.
\bibitem{Erdmann/Holm/Iyama/Schroer:2004a} K. Erdmann, T. Holm, O. Iyama and J. Schr\"{o}er, Radical embeddings and representation dimension, Adv. Math. 185 (2004), no. 1, 159--177.
\bibitem{Erdmann/Lim:2025} K. Erdmann and K. J. Lim, The representation type of the descent algebras of type $\A$, Sci. China Math. 68 (2025), no. 12, 2827--2846.
\bibitem{Gabriel:1972a} P. Gabriel, Unzerlegbare Darstellungen. I. Manuscripta Math 6, 71--103 (1972). correction, ibid. 6 (1972), 309.
\bibitem{Gabriel:1979a} P. Gabriel, Auslander-Reiten sequences and representation-finite algebras. Representation theory I (Proc. Workshop, Carleton Univ., Ottawa, Ont., 1979), pp. 1--71. Lecture Notes in Math., 831, Springer, Berlin, 1980.
\bibitem{Garsia/Reutenauer:1989a} A. M. Garsia and C. Reutenauer, A decomposition of Solomon's descent algebra, Adv. Math. 77 (1989), no. 2, 189--262.
\bibitem{Geck/Pfeiffer:2000} M. Geck and G. Pfeiffer, Characters of Finite Coxeter Groups and Iwahori–Hecke Algebras, London Math. Soc. Monographs, vol. 21, Oxford University Press, 2000.
\bibitem{Humphreys:1990} J. E. Humphreys, Reflection Groups and Coxeter Groups, Cambridge Stud. Adv. Math., 29 Cambridge University Press, Cambridge, 1990. 


\bibitem{Krause:1997} H. Krause, Stable equivalence preserves representation type, Comment. Math. Helv. 72 (1997), no. 2, 266--284.

\bibitem{Lim:2023a} K. J. Lim, Modular idempotents for the descent algebras of type $\A$ and higher Lie powers and modules, J. Algebra, 628 (2023), 98--162.
\bibitem{Ringel:1975a} C. M. Ringel, The Representation Type of Local Algebras, Lecture Notes in Mathematics, vol. 488, Springer, Berlin, 1975.
\bibitem{Saliola:2008a} F. V. Saliola, On the quiver of the descent algebra, J. Algebra 320 (2008), 3866--3894.
\bibitem{Schocker:2004a} M. Schocker, The descent algebra of the symmetric group, in: Representations of Finite Dimensional Algebras and Related Topics in Lie Theory and Geometry, Fields Institute Communications, 40 (American Mathematical Society, Providence, RI, 2004), 145--161.
\bibitem{Solomon:1976a} L.  Solomon, A Mackey formula in the group ring of a Coxeter group, J. Algebra 41 (1976) 255--268.
\end{thebibliography}
\end{document}